\documentclass[reqno]{amsart}
\usepackage{amsfonts}
\usepackage{amssymb,latexsym}
\usepackage{amsmath,color}
\usepackage{amsthm}
\usepackage{epsfig}
\usepackage{graphicx}
\usepackage{hyperref}
\usepackage{titletoc}
\usepackage{fullpage}
\usepackage{palatino,mathpazo}
\numberwithin{equation}{section}
\newtheorem{theorem}{Theorem}[section]

\newtheorem{lemma}[theorem]{Lemma}

\theoremstyle{definition}

\newtheorem{remark}{Remark}[section]

\def\XXint#1#2#3{{\setbox0=\hbox{$#1{#2#3}{\int}$}
     \vcenter{\hbox{$#2#3$}}\kern-.5\wd0}}

\newcommand{\ba}{\begin{array}}
    \newcommand{\ea}{\end{array}}

\newcommand{\ds}{\displaystyle}

\begin{document}

\title[Liouville theorems for $\chi$-type system]{Liouville type theorems and periodic solutions for $\chi^{(2)}$ type systems with non-homogeneous nonlinearities}

\author[A. Jevnikar]{Aleks Jevnikar}
\address{Aleks Jevnikar, Department of Mathematics, Computer Science and Physics,
    University of Udine, Via delle Scienze 206, 33100 Udine, Italy}
\email{aleks.jevnikar@uniud.it}

\author[J. Wang]{Jun Wang}
\address{ Jun ~Wang,~School of Mathematical Science, Jiangsu University, Zhenjiang,
Jiangsu, 212013, P.R. China} \email{wangmath2011@126.com}

\author[W. Yang]{ Wen Yang}
\address{\noindent Wen ~Yang,~Wuhan Institute of Physics and Mathematics, Chinese Academy of Sciences, P.O. Box 71010, Wuhan 430071, P. R. China; Innovation Academy for Precision Measurement Science and Technology, Chinese Academy of Sciences, Wuhan 430071, P. R. China.}
\email{wyang@wipm.ac.cn}

\begin{abstract}
In the present paper we derive Liouville type results and existence of periodic solutions for $\chi^{(2)}$ type systems with non-homogeneous nonlinearities. Moreover, we prove both universal bounds as well as singularity and decay estimates for this class of problems. In this study, we have to face new difficulties due to the non-homogenous nonlinearities. To overcome this issue, we carry out delicate integral estimates for this class of nonlinearities and modify the usual scaling and blow up arguments. This seems to be the first result for parabolic systems with non-homogeneous nonlinearities.

\


\end{abstract}

\maketitle

\section{Introduction}
In the present paper we consider the following parabolic systems
with non-homogeneous nonlinearities \begin{equation}\label{auto-1}
\begin{cases}
\ds u_t-\Delta
u=\mu_1u^p+\beta uv,\ &x\in\Omega,\ t\in(0, T),\\
\ds v_t-\Delta v=\mu_2v^p+\frac{\beta}{2} u^2,\ &x\in\Omega,\
t\in(0, T),
\end{cases}
\end{equation}
and
\begin{equation}\label{auto-2}
\begin{cases}
\ds u_t-\Delta
u=\mu_1u^p+\beta vw,\ &x\in\Omega,\ t\in(0, T),\\
\ds v_t-\Delta v=\mu_2v^p+\beta uw,\ &x\in\Omega,\ t\in(0, T),\\
\ds w_t-\Delta w=\mu_3w^p+\beta uv,\ &x\in\Omega,\ t\in(0, T),
\end{cases}
\end{equation}
where $p>1$ and $\Omega$ is a smooth domain in
$\mathbb{R}^{N}(1\leq N\leq5)$. The later system \eqref{auto-2} can be seen as a parabolic counterpart of
the three coupled nonlinear Schr\"odinger system below
\begin{equation}
\label{aj-1}
\begin{cases}
\frac{1}{\sqrt{-1}}\frac{\partial u_{1}}{\partial t}=-\Delta u_{1}-\mu_{1}|u_{1}|^{p-1}u_{1}-\beta\overline{u_{2}}u_{3},\\
\frac{1}{\sqrt{-1}}\frac{\partial u_{2}}{\partial t}=-\Delta u_{2}-\mu_{2}|u_{2}|^{p-1}u_{2}-\beta\overline{u_{1}}u_{3},\\
\frac{1}{\sqrt{-1}}\frac{\partial u_{3}}{\partial t}=-\Delta
u_{3}-\mu_{3}|u_{3}|^{p-2}u_{3}-\beta u_{1}u_{2},
\end{cases}
\end{equation}
where $u_{i}(i=1,2,3)$ are complex valued functions of
$(t,x)\in\mathbb{R}\times\mathbb{R}^{N}$, $p>1$, $N\leq3$,
$\mu_{i}>0(i=1,2,3)$ and $\beta\in\mathbb{R}$. System \eqref{aj-1} is a reduced system studied in
\cite{Colin-2004-Funkcial,Colin-2009-Annals,Mathieu-Colin-2012-SIAM,Pomponio-2010-JMP}
and related to the Raman amplification in a plasma. It is an
unstable phenomenon happening when an incident laser field
propagates into a plasma. This kind of model was first introduced by Russell et al. \cite{Russell-1999-PhyPlasmas} to describe the Raman scattering in a plasma. In the paper \cite{Colin-2004-DIF}, a modified model was derived to describing nonlinear interaction between a laser beam and a plasma. From the physical point
of view, when an incident laser field enters a plasma, it is
backscattered by a Raman type process. These two waves interact to yield an electronic plasma wave. The three waves combine to generate variation of the density of the ions which has itself an influence on the three proceedings waves. The system describing this phenomenon is composed by three Schr\"odinger equations coupled to a wave equation and reads in a suitable dimensionless form. In fact, system
\eqref{aj-1} originates from the following so-called $\chi^{(2)}$ system, see \cite{Buryak-2002-PhysRep} for the derivation of this system and its background
\begin{equation}
\label{aj-2}
\begin{cases}
\frac{1}{\sqrt{-1}}\frac{\partial u_1}{\partial t}+\Delta u_1-\lambda_1 u_1+\beta\overline{u_2}u_3=0,\quad &x\in\mathbb{R}^N,\\
\frac{1}{\sqrt{-1}}\frac{\partial u_2}{\partial t}+\Delta u_2-\lambda_2 u_2+\beta u_3\overline{u_1}=0,\quad &x\in\mathbb{R}^N,\\
\frac{1}{\sqrt{-1}}\frac{\partial u_3}{\partial t}+\Delta u_3-\lambda_3u_3+\beta u_1u_2=0,\quad &x\in\mathbb{R}^N.\\
\end{cases}
\end{equation}
For the purpose of modeling the nonlinear effects, one replaces the
linear terms $\lambda_iu_i$ in \eqref{aj-2} by
$\mu_i|u_i|^{p-1}u_i$. For a complete description of the model \eqref{aj-1}
(as well as a precise description of the physical coefficients) and
the $\chi^{(2)}$ system \eqref{aj-2}, we refer the readers to
\cite{Buryak-2002-PhysRep,Colin-2004-DIF,Colin-2004-JCAM,Colin-2011-ESAIM,Colin-2009-Annals,Zhao-Shi-2015-CVPDE}
and the references therein. A system similar to \eqref{aj-1} also
appears as an optics model with quadratic nonlinearity, see
\cite{Yew-2000-IUMJ,Yew-2001-JDE}.

Equations \eqref{auto-1}-\eqref{auto-2} can be seen as a special form of the following general parabolic system
\begin{equation}\label{a-j3}
\begin{cases}
\frac{\partial u_1}{\partial_t}-\Delta u_1=f_1(x, u_1,u_2,u_3),\   \ &(x,t)\in \Omega\times(T_0,\infty),\\
\frac{\partial u_2}{\partial_t}-\Delta u_2=f_2(x, u_1,u_2,u_3),\   \ &(x,t)\in \Omega\times(T_0,\infty),\\
\frac{\partial u_3}{\partial_t}-\Delta u_3=f_3(x, u_1,u_2,u_3),\   \
&(x,t)\in \Omega\times(T_0,\infty),
\end{cases}
\end{equation}
where $\Omega\subseteq\mathbb{R}^{N}$ is an open set. The system
\eqref{a-j3} has been studied in various mathematical
directions. For example, local and global existence
\cite{Amann-1985-Cerell,Morgan-1989-SIAM}, H\"older regularity
\cite{Dancer-Kelei-2011-JDE}, symmetry properties
\cite{Foldes-Polacik-2009}, blow-up behavior \cite{Merle-Zaag-2000},
and Liouville type theorems
\cite{Foldes-Polacik-2009,Phan-2005-DCDS,Souplet-2016-Math-Ann,Quittner-Mathann-2016,Quittner-JDE-2016}. Among all conclusions the Liouville type results has been very much investigated in the past decade and such result could be used to describe the asymptotic behavior of
the solution to \eqref{a-j3} for long time behavior. One purpose of  this paper is to prove Liouville type theorems for the
problems \eqref{auto-1}-\eqref{auto-2} and use them to show several qualitative behavior of the solutions.

\medskip

In order to introduce our results we first recall the known facts
for the  single parabolic equation
\begin{equation}\label{a-j4}
u_t-\Delta u=u^p,\ (x, t)\in\Omega\times(0, T),
\end{equation}
where $\Omega\subseteq\mathbb{R}^{N}$ is any smooth domain (bounded or unbounded). By
modifying the technique of local, integral gradient estimates
developed in \cite{Gidas-Spruck-CPAM-1981}, V\'eron in
\cite{Bidaut-1998-Initial} proved that the problem \eqref{a-j4} has no nontrivial non-negative classical solution for $1<p<p_B(N)$, where
$(x, t)\in\mathbb{R}^N\times\mathbb{R}$ and
\begin{equation}\label{a-j5}
p_B(N)=\begin{cases}
\frac{N(N+2)}{(N-1)^2},\quad&\text{if}\ N\geq2,\\
\infty, &\text{if}\ N=1.
\end{cases}
\end{equation}
When $N=1$, Pol\'a\u{c}ik and Quittner \cite{Quittner-NA-2006} gave a different proof of the latter result.  In the same paper, they also considered the radial solution and prove that the problem \eqref{a-j4} has no nontrivial
non-negative radial bounded classical solution for $1<p<p_S(N)$ by
using the arguments of intersection comparison with (sign-changing)
stationary solutions, see \cite{Souplet-ADV-2003} for an earlier partial result. Here $p_S(N)$ refers to the Sobolev critical exponent
\begin{equation}
\label{a.sobolev}
p_S(N)=
\begin{cases}
\frac{N+2}{N-2},\quad&\text{if}\ N\geq3,\\
\infty, &\text{if}\ N=1,2.
\end{cases}
\end{equation}
As a consequence of such Liouville type results, Souplet et al. \cite {Souplet-IUMJ-2007} obtained singularity and
decay estimates for the problem \eqref{a-j4} with general
nonlinearity and proved \eqref{a-j4} has no nontrivial non-negative radial classical solution for $1<p<p_S(N)$. In the radial case, the Liouville type result is optimal, since it is well-known that, for
$N\geq3$ and $p\geq p_S(N)$, \eqref{a-j4} admits positive stationary
solutions which are radial and bounded. Based on the above result, it is natural to state the following conjecture:
\medskip

\noindent {{\bf Conjecture:}} The problem \eqref{a-j4} has no nontrivial nonnegative classical
solution for $1<p<p_S(N)$.
\medskip

\noindent One of the main difficulties is that the techniques of moving planes or moving spheres can not be adapted to this problem as in the elliptic case \cite{Wolfgang-Zou-2000}. Recently, there are many works
towards proving the above conjecture and studying the same problem for system and other related models. Concerning the above conjecture, Quittner \cite{Quittner-Mathann-2016} first give a affirmative answer in $N=2$ and then completely solve the above problem in a recent paper \cite{Quittne-2020-Preprint} through the delicate analysis and careful study of the associated energy. 

Moving to parabolic system the situation is much more involved and one typically needs to make some assumptions on the nonlinearities. In this direction, the following problem has attracted a lot attentions in past decade,
\begin{equation}
\label{a-9}
U_t-\Delta U=F(U),\quad (x,t)\in\mathbb{R}^{N}\times\mathbb{R},
\end{equation}
where $F$ satisfies the homogeneous condition $F(\lambda
U)=\lambda^kF(U)$ for $\lambda\in(0,\infty)$ and
$U=(u_1,\cdots,u_m)\in[0,\infty)^m\setminus\{0\}(m, k\in\mathbb{N})$. One can see
\cite{Phan-2018-DCDS,Merle-Zaag-2000,Phan-2005-DCDS,Souplet-2016-Math-Ann,Quittner-Mathann-2016,Quittner-DCDS-2011}
and the references therein for the recent developments. In particular, by making full use of the homogeneous condition on the nonlinearity $F$, Quittner \cite{Quittner-Mathann-2016} proved that the system \eqref{a-9} does not possess any nontrivial non-negative solutions when $p>1$ and $(N-2)p<N$ by modifying the arguments of \cite{Bartsch-Quittner-JEMS-2011,Fila-Souplet-2001-Mathann,Giga-1987-IUMJ}. It is worth pointing out that this results is optimal when $N\leq2$. Immediately after \cite{Quittner-Mathann-2016}, Phan and Souplet \cite{Souplet-2016-Math-Ann} considered the system \eqref{a-9} with
\begin{equation}\label{a-10}
\begin{split}
&F(U)=\left(\sum_{j=1}^m\beta_{1j}u_1^ru_j^{r+1},\cdot\cdot\cdot,
\sum_{j=1}^m\beta_{mj}u_1^ru_j^{r+1}\right)\quad\text{and}\\
&\beta_{ij}\geq0\quad \text{for\ all}\quad i\neq
j\quad\text{and}\quad \beta_{ii}>0\ \text{for\ all}\ i.
\end{split}
\end{equation}
They proved that the system \eqref{a-9} has no nontrivial
non-negative classical solution in $\mathbb{R}^N\times\mathbb{R}$ when $N\geq3$ and $1<p=2r+1<p_B(N)$. Later on, Duong and Phan \cite{Phan-2018-DCDS} studied the following coupled system
\begin{equation}\label{a-11}
\begin{cases}
\ds u_t-\Delta
u=a_{11}u^p+a_12u^rv^{s+1},\ &x\in\Omega,\ t\in(0, T),\\
\ds v_t-\Delta v=a_{22}v^p+a_21u^{r+1}v^s,\ &x\in\Omega,\ t\in(0,
T),
\end{cases}
\end{equation}
where $p=r+s+1$. Observe that \eqref{a-11} is homogenous. They proved  similar Liouville
type results of \eqref{a-11} for $r\neq s$. For the Lotka-Volterra parabolic problem, another interesting two component system from the Mathematical Biology is
\begin{equation}\label{a-12}
\begin{cases}
\ds u_t-d_1\Delta
u=u(a_1(x,t)-b_1(x,t)u+c_1v),\ &x\in\Omega,\ t\in(0, T),\\
\ds v_t-d_2\Delta v=v(a_2(x,t)-b_2(x,t)u+c_2v),\ &x\in\Omega,\
t\in(0, T),
\end{cases}
\end{equation}
where $\Omega$ is a (possibly unbounded) domain in $\mathbb{R}^{N}$ with a uniformly $C^2$ smooth boundary $\partial\Omega$, $d_1, d_2$ are positive constants and $a_i, b_i, c_i\in L^\infty(\Omega\times(0,\infty))$ satisfies $b_i, c_i>0(i=1,2)$ and $c_1c_2>b_1b_2$. Quittner \cite{Quittner-JDE-2016} established Liouville theorems,
universal estimates and existence of periodic solutions for \eqref{a-12}.

\medskip

Motivated by the above works, see in particular
\cite{Souplet-2016-Math-Ann,Souplet-IUMJ-2007,Quittner-Mathann-2016,Quittner-JDE-2016}, we will pursue here the same research program for the parabolic systems with non-homogeneous nonlinearities \eqref{auto-1} and \eqref{auto-2}. More precisely, the purpose of the present paper is three-fold: first,
we derive Liouville type results in whole or half space. Second,
we prove singularity and decay estimates. Finally, we obtain the existence of periodic solutions. The case with the coefficients of diagonal terms in \eqref{auto-1}-\eqref{auto-2} being negative is also treated. In this study, there are some new challenging difficulties due to the presence of a non-homogenous structure. In particular, the arguments of
\cite{Phan-2018-DCDS,Souplet-2016-Math-Ann,Quittner-Mathann-2016} can not be adapted to this class of problems directly. To overcome this issue, we use the idea of
\cite{Bidaut-1998-Initial} to obtain first integral estimates for this type of nonlinearities. Then we prove Liouville type results by modifying the scaling and blow up arguments 
in \cite{Souplet-2016-Math-Ann,Souplet-IUMJ-2007}. To the best of our knowledge, this seems to be the first result for parabolic systems with non-homogeneous nonlinearities.

\medskip

Before going on, it is also interesting to mention some results on the related elliptic counterpart of \eqref{auto-1} and \eqref{auto-2}. By using variational methods, the paper \cite{Pomponio-2010-JMP} proved that the elliptic counterpart of \eqref{auto-1} has a positive least energy solution when
$3<p<2^{*}$ and $\beta>0$ sufficiently large. Previously, Shi and the second author of this paper obtain the existence of multiple positive solutions in the case $p=3$, $\lambda_1, \lambda_2$ and $\lambda_3$ pairwisely distinct. While for the case $2<p<3$, the second author of this paper proved the existence of positive solutions by combining the Mountain-Pass theorem in convex sets and
the Nehari constraint methods in  \cite{Wang-CVPDE-2017}. One can see the recent papers \cite{Wang-Shi-2016-Liuville,Wang-Shi-2019-Classification} on the Liouville type results and classification of the nontrivial synchronous solutions for such elliptic system.

\subsection{Two coupled system}
Our first result is about Liouville type theorems, singularity and decay estimates for the two coupled system \eqref{auto-1}.

\begin{theorem}
\label{th1.1}
\begin{itemize}
\item [$(i)$] Suppose that $\mu_1,\mu_2,\beta>0$, $N\leq 5$, $1<p<p_B(N)$, $\Omega=\mathbb{R}^N$ and $t\in\mathbb{R}$. Then the system \eqref{auto-1} does not possess any nontrivial non-negative classical solution.

\item [$(ii)$] Let $\Omega$ be any smooth domain of $\mathbb{R}^{N}$, $N\leq 5$, $2<p<p_B(N)$ and $t\in(0, T)$. For any non-negative nontrivial solution $(u,v)$ of \eqref{auto-1} on $\Omega\times(0,T)$, we have
\begin{equation}
\label{a9}
u(x,t)+v(x,t)\leq C\left(C_1+t^{-\frac{1}{p-1}}+(T-t)^{-\frac{1}{p-1}}+(dist(x,\partial\Omega))^{-\frac{2}{(p-1)}}\right)
\end{equation}
for $(x,t)\in\Omega\times(0, T)$ with some positive constant $C$ and non-negative constant $C_1$. Particularly, $C_1=0$ if $p=2$. Moreover, we have
\begin{equation*}
(T-t)^{-1}=0\ \text{if}~\ T=\infty\quad \text{and}\quad
(dist(x,\partial\Omega))^{-\frac{2}{(p-1)}}=0~\ \text{if}~\ \Omega=\mathbb{R}^{N}.
\end{equation*}
\end{itemize}
\end{theorem}

\begin{remark}\label{rek-2.1}
For single heat equation with algebraic nonlinearity, V\'eron in \cite{Bidaut-1998-Initial} proved Liouville type results via the Harnack inequality for the single equation. Due to the strong coupling interaction, it is not easy to apply such techniques for system straightforwardly. Instead, for a class of systems in \cite{Souplet-2016-Math-Ann,Souplet-IUMJ-2007,Souplet-Duke-2007} the authors make use of homogeneity type arguments exploiting the homogeneous structure of the nonlinear terms. Even though our system does not admit homogeneous properties, we are still able to combine the method of V\'eron and scaling techniques to achieve the same goal. Roughly speaking, this is due to the structure of the system \eqref{auto-1} and the cross nonlinear term can be controlled together with the diagonal nonlinear term, see the key result in Lemma \ref{lem-2.3}. While for Theorem \ref{th1.1}-$(ii)$, we can only get the a-priori estimates for $p\geq2$. The reason is that the following homogeneous quadratic system
\begin{equation*}
u_t-\Delta u=\beta uv,\quad v_t-\Delta v=\beta/2u^2,\ x\in\Omega,\
t\in(0, T),
\end{equation*}
has semi-trivial solutions $(0, C_0)$ with $C_0$ being any constant.
\end{remark}

Next we study the following half-space problem
\begin{equation}\label{auto-10}
\begin{cases}
\ds u_t-\Delta
u=\mu_1u^p+\beta uv,\ &x\in\mathbb{R}_+^{N},\ t\in\mathbb{R},\\
\ds v_t-\Delta v=\mu_2v^p+\frac{\beta}{2} u^2,\
&x\in\mathbb{R}_+^{N},\ t\in\mathbb{R},\\
u=v=0, &x\in\partial\mathbb{R}_+^{N},\ t\in\mathbb{R},
\end{cases}
\end{equation}
where $\mathbb{R}_+^{N}=\{x\in\mathbb{R}^{N}: x_1>0\}$. Following the idea of
\cite{Souplet-2016-Math-Ann,Souplet-IUMJ-2007} we can show that any non-negative bounded solution of \eqref{auto-10} is increasing in $x_1$, i.e., $\partial u/\partial x_1>0$ and $\partial v/\partial
x_1>0$. Combined with the scaling technique we have the following result.

\begin{theorem}
\label{th1.2}
Suppose that $\mu_1, \mu_2, \beta>0$, $N\leq 5$ and $1<p<\frac{N(N+2)}{(N-1)^2}$. Then the problem \eqref{auto-10} does not possess any nontrivial non-negative bounded classical solution.
\end{theorem}

We also study the $T$-periodic solutions of the system
\begin{equation}\label{auto-11}
\begin{cases}
\ds u_t-\Delta
u=\mu_1u^p+\beta uv,\ &x\in\Omega,\ t\in(0, \infty),\\
\ds v_t-\Delta v=\mu_2v^p+\frac{\beta}{2} u^2,\ &x\in\Omega,\
t\in(0, \infty),\\
u(x,t)=v(x,t)=0,\quad &x\in\partial\Omega,\ t\in\times(0, \infty),\\
u(x,0)=u(x,T)\ \text{and}\ v(x,0)=v(x,T), &x\in\Omega,
\end{cases}
\end{equation}
where $\Omega\subset\mathbb{R}^N(N\leq5)$ is a smooth bounded domain. Then we have the following conclusion.

\begin{theorem}
\label{th1.3}
Suppose that $\mu_1, \beta>0$, $N\leq 5$ and $2\leq p<p_B(N)$. Then the problem \eqref{auto-11} has a positive $T$-periodic solution provided $\mu_2\geq0$ is sufficiently small.
\end{theorem}

\begin{remark}
Our method of proving Theorem \ref{th1.2} is based on the degree theory method. We can only apply it to prove the existence of $T$-periodic positive solution of \eqref{auto-10} for $p\geq2$. While for the case $1<p<2$, due to the same reason as in Remark \ref{rek-2.1}, we can not establish the boundedness of $T$-periodic solutions and this prevents us from using the degree theory method.
\end{remark}

The last result concerning \eqref{auto-1} is about Liouville type theorems and existence of periodic solution for the case when the coefficients $\mu_1, \mu_2$
are negative.

\begin{theorem}
\label{th1.4}
Let $\mu_1, \mu_2<0$ and $\beta>0$. Then the following
results hold.
\begin{itemize}
\item [(1)] If $p=2,~N\leq5$ and $3|\mu_1|^2<2\beta(\beta+|\mu_2|)$, then the problem \eqref{auto-1} has no non-negative solution when $(x,t)\in\mathbb{R}^{N}\times\mathbb{R}$. In addition, for the case $\Omega=\mathbb{R}_+^{N}$ and $t\in\mathbb{R}$, let $(u,v)$ be a non-negative bounded solution of \eqref{auto-1} equiped with the homogeneous Dirichlet boundary condition. Then $u=v=0$.

\item [(2)] If $p=2,~N\leq5$ and $3|\mu_1|^2<2\beta(\beta+|\mu_2|)$, then the problem \eqref{auto-11} has a positive $T$-periodic solution.
\end{itemize}
\end{theorem}

\begin{remark}
In the proof of Theorem \ref{th1.4}, we first show that there exists
$M>0$ such that any non-negative(or non-negative bounded) solution $u,v$ verifies that $u=Mv$. This implies that $u$ is a solution of the scalar
equation
\begin{equation}
\label{a-16}
u_t-\Delta
u=\left(\sqrt{|\mu_1|^2+2\beta(\beta+|\mu_2|)}-2|\mu_1|\right)u^2,\quad
(x,t)\in\Omega\times \mathbb{R}.
\end{equation}
If
$\sqrt{|\mu_1|^2+2\beta(\beta+|\mu_2|)}-2|\mu_1|>0$, then we can see that \eqref{a-16} has no non-negative solutions. While for the case
$p\neq2$, we derive that $u$ satisfies
\begin{equation}\label{a-17}
u_t-\Delta u=\mu_1u^p+\beta u^2,\quad (x,t)\in\Omega\times
\mathbb{R},
\end{equation}
where $\mu_1<0$ and $p\neq2$. Up to now we do not know any
Liouville type results of \eqref{a-17}, due to the non-homogeneous nonlinearity and the coefficients of the diagonal nonlinear term being negative. This is an interesting issue that we shall pursue in the future.
\end{remark}

\subsection{Three coupled system}
In this part we focus on the three coupled system \eqref{auto-2}. We first present the following Liouville type result together with singularity and decay estimates.
\begin{theorem}
\label{th1.6}
\begin{itemize}
\item [$(i)$] Suppose that $\mu_1,\mu_2,\beta>0$, $N\leq 5$, $1<p<p_B(N)$, $\Omega=\mathbb{R}^N$ and $t\in\mathbb{R}$. Then the system \eqref{auto-2} has no nontrivial non-negative classical solution.

\item [$(ii)$] Let $\Omega$ be any smooth domain of $\mathbb{R}^{N}$, $N\leq 5$, $2\leq p<p_B(N)$ and $t\in(0,T)$. For any non-negative nontrivial solution $(u,v,w)$ of \eqref{auto-2} on $\Omega\times(0,T)$, then we have
\begin{equation}
\label{aj9}
u(x,t)+v(x,t)+w(x,t)\leq C\left(C_1+t^{-\frac{1}{p-1}}+(T-t)^{-\frac{1}{p-1}}+(dist(x,\partial\Omega))^{-\frac{2}{(p-1)}}\right)
\end{equation}
for $(x,t)\in\Omega\times(0, T)$ with some positive constant $C$ and non-negative constant $C_1$. Particularly, $C_1=0$ if $p=2$. Moreover, we have
\begin{equation*}
(T-t)^{-1}=0\ \text{if}~\ T=\infty\quad \text{and}\quad
(dist(x,\partial\Omega))^{-\frac{2}{(p-1)}}=0~\ \text{if}~\ \Omega=\mathbb{R}^{N}.
\end{equation*}
\end{itemize}
\end{theorem}

As Theorem \ref{th1.1}, we can only get the a-priori estimates of \eqref{auto-2} for $p\geq2$ in Theorem \ref{th1.6}-$(ii)$, the reason is that the homogeneous quadratic system
\begin{equation*}
u_t-\Delta u=\beta wv,\ v_t-\Delta v=\beta uw, \ w_t-\Delta w=\beta
uv,\ x\in\Omega,\ t\in(0, T)
\end{equation*}
has three semitrivial solutions $(0, 0, C_0), (0, C_0, 0)$ and $(C_0, 0)$, where $C_0\in\mathbb{R}$ is constant.

Finally, we study the existence of periodic solutions to the three coupled system
\begin{equation}\label{auto-13}
\begin{cases}
\ds u_t-\Delta
u=\mu_1u^p+\beta vw,\ &x\in\Omega,\ t\in(0, \infty),\\
\ds v_t-\Delta v=\mu_2v^p+\beta uw,\ &x\in\Omega,\
t\in(0, \infty),\\
\ds w_t-\Delta w=\mu_3w^p+\beta uv,\ &x\in\Omega,\ t\in(0, \infty),\\
u(x,t)=v(x,t)=w(x,t)=0,\quad &x\in \partial\Omega,\ t\times(0,
\infty),
\end{cases}
\end{equation}
where $u(x,0)=u(x,T)$, $v(x,0)=v(x,T)$, $\text{and}$ $w(x,0)=w(x,T)$
for $x\in\Omega$, $\Omega\subset\mathbb{R}^N(N\leq5)$ is any smooth bounded domain. We have the following conclusion.

\begin{theorem}
\label{th1.7}
Suppose that $\beta>0$ and $2\leq p<p_B(N)(N\leq5)$. Then the problem \eqref{auto-13} possesses a positive $T$-periodic solution when $\mu_1, \mu_2, \mu_3>0$ are sufficiently small.
\end{theorem}

The outline of the rest of the paper is as follows. In Section 2 we
prove Liouville type results together with singularity and decay estimates for two coupled system \eqref{auto-1} in the whole
or half space. Section 3 is devoted in proving the existence of
periodic solutions for \eqref{auto-10}. Finally, in Section 4 we study
Liouville type results, singularity and decay estimates and prove the existence of periodic solutions
of the three coupled system \eqref{auto-2}.

\medskip
\begin{center}
Notation:
\end{center}

\indent Throughout the paper, the letter C will stand for positive constants which are allowed to vary among different formulas and even within the same lines. $\mathcal{D}(\Omega)$ denotes the set of all smooth function with compact support in $\Omega$.
\bigskip

\section{The Liouville type results}
\subsection{The Liouville results in the whole space}
In this subsection we consider the two coupled system \eqref{auto-1} and give the proof of Theorem \ref{th1.1}, that is Liouville type results and singularity and decay estimates. The starting point is the following inequality.

\begin{lemma}
\label{lem-2.1}
Let $\Omega\subset\mathbb{R}^{N}$ be any open set and $d,m$ are two constants such that $d\neq m+2.$ Then for any
function $u\in C^2(\Omega)$ and any nonnegative $\xi\in\mathcal
{D}(\Omega)$,
\begin{equation}
\label{b-1}
\begin{split}
&\frac{2(N-m)d-(N-1)(m^2+d^2))}{4N}\int_{\Omega}\xi u^{m-2}|\nabla
u|^4dx-\frac{N-1}{N}\int_{\Omega}\xi u^m(\Delta u)^2dx\\
&\quad -\frac{2(N-1)m+(N+2)d}{2N}\int_{\Omega}\xi u^{m-1}|\nabla
u|^2\Delta udx\\
&\leq\frac{m+d}{2}\int_{\Omega}u^{m-1}|\nabla u|^2\nabla
u\cdot\nabla\xi dx+\int_{\Omega}u^m\Delta u\nabla
u\cdot\nabla\xi dx+\frac{1}{2}\int_{\Omega}u^m|\nabla u|^2\Delta\xi dx.
\end{split}
\end{equation}
\end{lemma}

\begin{proof}
The proof is based on the following well-known B\"ochner-Wietzenb\"ock formula
$$\frac12\Delta\left(|\nabla f|^2\right)=|\mathrm{Hess}f|^2+(\nabla\Delta f)\nabla f
\geq\frac1N(\Delta f)^2+(\nabla\Delta f)\nabla f,$$
where $\mathrm{Hess}f$ is the Hessian of $f$. Substituting $f=u^{\frac{m+2-d}{2}}$ into the above inequality and multiplying the inequality by $(\frac{m+2-d}{2})^{-2}\xi u^d$. After using the integration by parts we get \eqref{b-1}. We refer the readers to \cite[Lemma 3.1]{Bidaut-Raoux-1996-CPDE} for the details.
\end{proof}

With Lemma \ref{lem-2.1}, we follow the idea of \cite{Bidaut-1998-Initial,Souplet-2016-Math-Ann} to derive a-priori estimates for the nonlinearities of \eqref{auto-1}. Let $(u,v)\in C^{2,1}(Q\times(-T_1,T_2))$ be a positive solution of \eqref{auto-1}, and $Q$ be any domain of $\mathbb{R}^{N}$ such that $G=\Omega\times(t_1, t_2)\subset\overline{G}\subset Q\times(-T_1,T_2)$, where $T_1, T_2>0$. By Lemma \ref{lem-2.1}, we give the essential interior estimate for the positive
solution $(u,v)$ of \eqref{auto-1}.

\begin{lemma}
\label{lem-2.2}
Assume that $\mu_1, \mu_2, \beta>0$ and $1<p<\frac{N(N+2)}{(N-1)^2}$. Let $\zeta\in\mathcal {D}(G)$ with value in $[0,1]$ and $q>4$. Then there exists a constant $C=C(N,p,q)>0$ such that
\begin{equation}
\label{b-2}
\begin{split}
&\int_{G}\zeta^{q}\left(u^{-2}|\nabla u|^4+v^{-2}|\nabla v|^4\right)dxdt
+\int_{G}\zeta^{q}\left(u^{p-1}|\nabla u|^2+v^{p-1}|\nabla v|^2\right)dxdt\\
&\quad+\int_{G}\zeta^q(u_t^2+v_t^2)dxdt+\int_{G}\zeta^{q}\left(|\nabla u|^2v+|\nabla v|^2u^2v^{-1}\right)dxdt\\
&\leq C\int_{G}\zeta^{q-2}\left(u^{p+1}+v^{p+1}\right)\left(|\zeta_t|+|\nabla\zeta|^2\right)dxdt+C\beta\int_{G}\zeta^{q-2}u^2v\left(|\zeta_t|+|\nabla\zeta|^2\right)dxdt\\
&\quad+C\int_{G}\zeta^{q-4}\left(u^2+v^2\right)\left(|\Delta\zeta|^2+|\nabla\zeta|^4+|\zeta_t|^2\right)dxdt.
\end{split}
\end{equation}
\end{lemma}

\begin{proof}
Using \eqref{b-1} for positive solution $(u,v)$ of \eqref{auto-1} with $m=0$, any constant $d\neq 2$ and $\xi=\zeta^q=\zeta^q(t,\cdot)$ for any $t\in[t_1,t_2]$, we get
\begin{equation}
\label{b-3}
\begin{split}
&A_0\int_{\Omega}\zeta^qu^{-2}|\nabla u|^4dx-\frac{N-1}{N}\int_{\Omega}\zeta^q(\Delta u)^2dx-
\frac{(N+2)d}{2N}\int_{\Omega}\zeta^qu^{-1}|\nabla u|^2\Delta udx\\
&\leq\frac{d}{2}\int_{\Omega}u^{-1}|\nabla u|^2\nabla
u\nabla(\zeta^q)dx+\int_{\Omega}\Delta u\nabla
u\nabla(\zeta^q)dx+\frac{1}{2}\int_{\Omega}|\nabla u|^2\Delta(\zeta^q)dx
\end{split}
\end{equation}
and
\begin{equation}
\label{b-4}
\begin{split}
&A_0\int_{\Omega}\zeta^qv^{-2}|\nabla v|^4dx
-\frac{N-1}{N}\int_{\Omega}\zeta^q(\Delta v)^2dx
-\frac{(N+2)d}{2N}\int_{\Omega}\zeta^qv^{-1}|\nabla v|^2\Delta vdx\\
&\leq\frac{d}{2}\int_{\Omega}v^{-1}|\nabla v|^2\nabla
v\nabla(\zeta^q)dx+\int_{\Omega}\Delta v\nabla
v\nabla(\zeta^q)dx+\frac{1}{2}\int_{\Omega}|\nabla
v|^2\Delta(\zeta^q)dx,
\end{split}
\end{equation}
where
$$A_0=\frac{\left(2Nd-(N-1)d^2\right)}{4N}.$$
Since the study of these two equations is the same, we shall only focus on the first one. From \eqref{auto-1} we get
\begin{equation}
\label{b-5}
\begin{split}
-\int_{\Omega}\zeta^q(\Delta u)^2dx&=-\int_{\Omega}\zeta^q\Delta
u(u_t-\mu_1u^p-\beta uv)dx\\
&=\frac{df_u}{dt}-P_u+\int_{\Omega}\nabla(\zeta^q)\nabla
uu_tdx-\mu_1p\int_{\Omega}\zeta^qu^{p-1}|\nabla
u|^2dx-\mu_1\int_{\Omega}\nabla(\zeta^q)\nabla
uu^{p}dx\\
&\quad-\beta\int_{\Omega}\nabla(\zeta^q)\nabla u
uvdx-\beta\int_{\Omega}\zeta^q|\nabla
u|^2vdx-\beta\int_{\Omega}\zeta^q\nabla u\nabla vudx,
\end{split}
\end{equation}
where
\begin{equation}
\label{b-6}
f_u=\frac{1}{2}\int_{\Omega}\zeta^q|\nabla u|^2dx
\quad\text{and}\quad
P_u=\frac{1}{2}\int_{\Omega}\left(\zeta^q\right)_t|\nabla u|^2dx.
\end{equation}
Similarly,
\begin{equation}
\label{b-7}
\begin{split}
-\int_{\Omega}\zeta^qu^{-1}|\nabla u|^2\Delta udx
&=-\int_{\Omega}\zeta^qu^{-1}|\nabla u|^2(u_t-\mu_1u^p-\beta uv)dx\\
&=-\int_{\Omega}\zeta^qu^{-1}|\nabla u|^2u_tdx
+\mu_1\int_{\Omega}\zeta^qu^{p-1}|\nabla u|^2dx
+\beta\int_{\Omega}\zeta^q|\nabla u|^2vdx,
\end{split}
\end{equation}
and
\begin{equation}
\label{b-7-1}
\begin{split}
\int_\Omega\nabla\zeta^q\cdot\nabla u\Delta udx
&=\int_\Omega\nabla\zeta^q\cdot\nabla u(u_t-\mu_1 u^p-\beta uv)dx\\
&=\int_\Omega\nabla\zeta^q\cdot\nabla u u_tdx-\mu_1\int_\Omega\nabla\zeta^q\cdot\nabla uu^pdx-\beta\int_\Omega\nabla\zeta^q\cdot\nabla uuvdx.
\end{split}
\end{equation}
Substituting \eqref{b-5}-\eqref{b-7-1} into \eqref{b-3}, we obtain
\begin{equation}
\label{b-8}
\begin{split}
&A_0\int_{\Omega}\zeta^qu^{-2}|\nabla u|^4dx
+B_1\int_{\Omega}\zeta^qu^{p-1}|\nabla u|^2dx
+C_1\int_{\Omega}\zeta^qv|\nabla u|^2dx-\frac{N-1}{N}Q\\
&\leq-\frac{N-1}{N}\left(\frac{df_u}{dt}-P_u\right)+\frac{(N+2)d}{2N}Y_u+\frac{1}{N}Z_u-\frac{1}{N}V_u
+\frac{d}{2}W_u+\frac{1}{2}R_u-\frac1NS_u,
\end{split}
\end{equation}
where
\begin{equation}
\label{b-9}
\begin{split}
&B_1=\frac{(N+2)d-2(N-1)p}{2N}\mu_1,\ C_1=\frac{(N+2)d-2(N-1)}{2N}\beta,\ Y_u=\int_{\Omega}\zeta^qu^{-1}|\nabla u|^2u_tdx,\\
&Z_u=\int_{\Omega}\nabla(\zeta^q)\nabla u u_tdx,\
V_u=\mu_1\int_{\Omega}u^p\nabla(\zeta^q)\nabla udx,\
W_u=\int_{\Omega}u^{-1}|\nabla u|^2\nabla u\nabla(\zeta^q)dx,\\
&R_u=\int_{\Omega}\Delta(\zeta^q)|\nabla u|^2dx,\
S_u=\beta\int_{\Omega}\nabla u\nabla(\zeta^q)uvdx\quad \text{and}\quad Q=\beta\int_{\Omega}\nabla u\nabla vu\zeta^qdx.
\end{split}
\end{equation}
Analogously,
\begin{equation}
\label{b-10}
\begin{split}
&A_0\int_{\Omega}\zeta^qv^{-2}|\nabla v|^4dx
+B_2\int_{\Omega}\zeta^qv^{p-1}|\nabla v|^2dx
+C_2\int_{\Omega}\zeta^qu^2v^{-1}|\nabla v|^2dx-\frac{N-1}{N}Q\\
&\leq-\frac{N-1}{N}\left(\frac{df_v}{dt}-P_v\right)+\frac{(N+2)d}{2N}Y_v+\frac{1}{N}Z_v-\frac{1}{N}V_v
+\frac{d}{2}W_v+\frac{1}{2}R_v-\frac1NS_v,
\end{split}
\end{equation}
where
\begin{equation}
\label{b-11}
\begin{split}
&\ B_2=\frac{(N+2)d-2(N-1)p}{2N}\mu_2, C_2=\frac{(N+2)d}{4N}\beta,
f_v=\frac{1}{2}\int_{\Omega}\zeta^q|\nabla v|^2dx,
P_v=\frac{1}{2}\int_{\Omega}\left(\zeta^q\right)_t|\nabla v|^2dx, \\
&Y_v=\int_{\Omega}\zeta^qv^{-1}|\nabla v|^2v_tdx,\
Z_v=\int_{\Omega}\nabla(\zeta^q)\nabla vv_tdx,\
V_v=\mu_2\int_{\Omega}v^p\nabla(\zeta^q)\nabla vdx,\\
&W_v=\int_{\Omega}v^{-1}|\nabla v|^2\nabla v\nabla(\zeta^q)dx,\
R_v=\int_{\Omega}\Delta(\zeta^q)|\nabla v|^2dx
\quad \text{and}\quad
S_v=\frac{\beta}{2}\int_{\Omega}\nabla v\nabla(\zeta^q)u^2dx.
\end{split}
\end{equation}
By H\"older's inequality,
\begin{equation}
\label{b-12}
\begin{split}
2Q=2\beta\int_{\Omega}\nabla u\nabla vu\zeta^qdx
\leq\beta\int_{\Omega}\zeta^q|\nabla u|^2vdx
+\beta\int_{\Omega}\zeta^q|\nabla v|^2u^2v^{-1}dx.
\end{split}
\end{equation}
From \eqref{b-8}, \eqref{b-10} and \eqref{b-12} we get
\begin{equation}
\label{b-13}
\begin{split}
&A_0\int_{\Omega}\zeta^q(u^{-2}|\nabla u|^4+v^{-2}|\nabla
v|^4)dx
+B_3\int_{\Omega}\zeta^q(u^{p-1}|\nabla u|^2+v^{p-1}|\nabla
v|^2)dx\\
&\quad+C_3\int_{\Omega}\zeta^q(v|\nabla u|^2+u^2v^{-1}|\nabla v|^2)dx\\
&\leq-\frac{N-1}{N}\left(\frac{df_u}{dt}-P_u+\frac{df_v}{dt}-P_v\right)+\frac{(N+2)d}{2N}(Y_u+Y_v)+\frac{1}{N}(Z_u+Z_v)\\
&\quad-\frac{1}{N}(V_u+V_v)+\frac{d}{2}(W_u+W_v)+\frac{1}{2}(R_u+R_v)-\frac1N(S_u+S_v),
\end{split}
\end{equation}
where
\begin{equation}
\label{b-14}
B_3=\min\{\mu_1,\mu_2\}\frac{(N+2)d-2(N-1)p}{2N}
\quad\text{and}\quad
C_3=\frac{(N+2)d-4(N-1)}{4N}\beta.
\end{equation}
Since $2\leq N\leq 5$ and $p<\frac{N(N+2)}{(N-1)^2}$, we can always choose $d>0$ such that
\begin{equation}
\label{b-15}
\max\left\{\frac{2(N-1)p}{N+2},\frac{4(N-1)}{N+2}\right\}<d<\frac{2N}{N-1}.
\end{equation}
Concerning the last six terms on the right-hand side of \eqref{b-13}, by Young's inequality, we have
\begin{equation}
\label{b-16}
\begin{split}
Y_u+Y_v&\leq \varepsilon\int_\Omega\zeta^q(u^{-2}|\nabla u|^4+v^{-2}|\nabla v|^4)dx+\frac{1}{4\varepsilon}\int_{\Omega}\zeta^q(u_t^2+v_t^2)dx,
\\
Z_u+Z_v&\leq\frac12\int_{\Omega}\zeta^q(u_t^2+v_t^2)dx+\varepsilon^2\int_{\Omega}\zeta^q(u^{-2}|\nabla u|^4+v^{-2}|\nabla v|^4)dx\\
&\quad+C(\varepsilon)\int_{\Omega}\zeta^{q-4}|\nabla\zeta|^4(u^2+v^2)dx,\\
V_u+V_v&\leq\varepsilon\int_{\Omega}\zeta^q(u^{p-1}|\nabla
u|^2+v^{p-1}|\nabla v|^2)dx+C(\varepsilon)\int_{\Omega}\zeta^{q-2}|\nabla \zeta|^2(u^{p+1}+v^{p+1})dx,\\
W_u+W_v&\leq\varepsilon\int_{\Omega}\zeta^q(u^{-2}|\nabla u|^4+v^{-2}|\nabla v|^4)dx+C(\varepsilon)\int_{\Omega}\zeta^{q-4}|\nabla\zeta|^4(u^2+v^2)dx\\
R_u+R_v&\leq\varepsilon\int_{\Omega}\zeta^q(u^{-2}|\nabla u|^4+v^{-2}|\nabla v|^4)dx
+C(\varepsilon)\int_{\Omega}\zeta^{q-4}|\nabla\zeta|^4(u^2+v^2)dx\\&\quad+C(\varepsilon)\int_{\Omega}\zeta^{q-2}|\Delta\zeta|^2(u^{2}+v^2)dx\\
S_u+S_v&\leq\varepsilon\beta\int_{\Omega}\zeta^q(v|\nabla u|^2+u^2v^{-1}|\nabla v|^2)dx+C(\varepsilon)\beta\int_{\Omega}\zeta^{q-2}|\nabla\zeta|^2u^2vdx,
\end{split}
\end{equation}
where $\varepsilon$ is arbitrarily small. 
From the original equation \eqref{auto-1}, we see that
\begin{equation}
\label{b-17}
\begin{split}
\int_{\Omega}\zeta^q(u_t^2+v_t^2)dx&=\int_{\Omega}\zeta^qu_t(\Delta u+\mu_1u^p+\beta uv)dx+
\int_{\Omega}\zeta^qv_t(\Delta v+\mu_2v^p+\frac{\beta}{2}u^2)dx
\\&=\frac{d(g_u-f_u)}{dt}+\frac{d(g_v-f_v)}{dt}+(P_u+P_v)-(Z_u+Z_v)-(K_u+K_v)+\frac{dh}{dt}-H,
\end{split}
\end{equation}
where
\begin{equation}
\label{b-18}
\begin{split}
&g_u=\frac{\mu_1}{p+1}\int_{\Omega}\zeta^qu^{p+1}dx,\
g_v=\frac{\mu_2}{p+1}\int_{\Omega}\zeta^qv^{p+1}dx,\
K_u=\frac{\mu_1}{p+1}\int_{\Omega}(\zeta^q)_tu^{p+1}dx,\\
&K_v=\frac{\mu_2}{p+1}\int_{\Omega}(\zeta^q)_tv^{p+1}dx,\
h=\frac{\beta}{2}\int_{\Omega}\zeta^qu^2vdx\quad\text{and}\quad
H=\frac{\beta}{2}\int_{\Omega}(\zeta^q)_tu^2vdx.
\end{split}
\end{equation}
Using Young's inequality for the terms $P_u$ and $P_v,$ we gain
\begin{equation}
\label{b-19}
P_u+P_v\leq\varepsilon^2\int_\Omega\zeta^q(u^{-2}|\nabla u|^4+v^{-2}|\nabla v|^4)dx+C(\varepsilon)\int_{\Omega}\zeta^{q-2}\zeta_t^2(u^2+v^2)dx.
\end{equation}
Combining \eqref{b-17}, \eqref{b-19} and the estimation on $Z_u+Z_v$ in \eqref{b-16}, we deduce that
\begin{equation}
\label{b-21}
\begin{split}
\frac{1}{2}\int_{\Omega}\zeta^q(u_t^2+v_t^2)dx&\leq\frac{d(g_u-f_u)}{dt}+\frac{d(g_v-f_v)}{dt}+
2\varepsilon^2\int_{\Omega}\zeta^q(u^{-2}|\nabla u|^4+v^{-2}|\nabla v|^4)dx+\frac{dh}{dt}\\
&\quad+C(\varepsilon)\int_{\Omega}(\zeta^{q-2}\zeta_t^2+\zeta^{q-4}|\nabla\zeta|^4)(u^2+v^2)dxdt
+\frac{\beta}{2}q\int_{\Omega}|\zeta_t|\zeta^{q-1}u^2vdx\\
&\quad+C\int_{\Omega}\zeta^{q-1}|\zeta_t|(u^{p+1}+v^{p+1})dx.
\end{split}
\end{equation}
Using the assumption $\zeta\in[0,1]$, \eqref{b-13}, \eqref{b-16} and \eqref{b-21},
\begin{equation}
\label{b-23}
\begin{split}
&(A_0-\varepsilon)\int_{\Omega}\zeta^q(u^{-2}|\nabla
u|^4+v^{-2}|\nabla v|^4)dx
+(B_3-\varepsilon)\int_{\Omega}\zeta^q(u^{p-1}|\nabla
u|^2+v^{p-1}|\nabla v|^2)dx
\\&\quad+(C_3-\varepsilon)\int_{\Omega}\zeta^q(v|\nabla u|^2+u^2v^{-1}|\nabla v|^2)dx\\
&\leq-\frac{N-1}{N}\frac{d(f_u+f_v)}{dt}+C(\varepsilon)\frac{d(g_u-f_u)}{dt}
+C(\varepsilon)\frac{d(g_v-f_v)}{dt}+C(\varepsilon)\frac{dh}{dt}\\
&\quad+C(\varepsilon)\beta\int_{\Omega}\zeta^{q-2}(|\zeta_t|+|\nabla\zeta|^2)u^2vdx
+C(\varepsilon)\int_{\Omega}\zeta^{q-2}(|\zeta_t|+|\nabla\zeta|^2)(u^{p+1}+v^{p+1})dx\\
&\quad+C(\varepsilon)\int_{\Omega}\zeta^{q-4}(|\Delta\zeta|^2+|\nabla\zeta|^4+|\zeta_t|^2)(u^2+v^2)dx.
\end{split}
\end{equation}
As a consequence,
\begin{equation}
\label{b-24}
\begin{split}
&\int_{\Omega}\zeta^q(u^{-2}|\nabla u|^4+v^{-2}|\nabla v|^4)dx
+\int_{\Omega}\zeta^q(u^{p-1}|\nabla u|^2+v^{p-1}|\nabla
v|^2)dx\\
&\quad+\int_{\Omega}\zeta^q(v|\nabla u|^2+u^2v^{-1}|\nabla v|^2)dx\\
&\leq
C\left(\frac{df_u}{dt}+\frac{df_v}{dt}+\frac{dg_u}{dt}+\frac{dg_v}{dt}+\frac{dh}{dt}\right)+
C\int_{\Omega}\zeta^{q-4}(|\Delta\zeta|^2+|\nabla\zeta|^4+|\zeta_t|^2)(u^2+v^2)dx\\
&\quad+C\beta\int_{\Omega}\zeta^{q-2}(|\zeta_t|+|\nabla\zeta|^2)u^2vdx+C\int_{\Omega}\zeta^{q-2}(|\zeta_t|+|\nabla\zeta|^2)(u^{p+1}+v^{p+1})dx.
\end{split}
\end{equation}
Integrating both sides from $t_1$ to $t_2,$ together with the fact that $\zeta\in\mathcal{D}(G)$, we get
\begin{equation}
\label{b-25}
\begin{split}
&\int_{G}\zeta^q(u^{-2}|\nabla u|^4+v^{-2}|\nabla
v|^4)dxdt+\int_{G}\zeta^q(u^{p-1}|\nabla u|^2+v^{p-1}|\nabla
v|^2)dxdt\\
&\quad+\int_{G}\zeta^q(v|\nabla u|^2+u^2v^{-1}|\nabla v|^2)dxdt\\
&\leq C\int_{G}\zeta^{q-4}(|\Delta\zeta|^2+|\nabla\zeta|^4+|\zeta_t|^2)(u^2+v^2)dxdt
+C\beta\int_{G}\zeta^{q-2}(|\zeta_t|+|\nabla\zeta|^2)u^2vdxdt\\
&\quad+C\int_{G}\zeta^{q-2}(|\zeta_t|+|\nabla\zeta|^2)(u^{p+1}+v^{p+1})dxdt.
\end{split}
\end{equation}
Finally, using \eqref{b-21} and \eqref{b-25} we derive that $\int_{\Omega}\zeta^q(u_t^2+v_t^2)dxdt$ can be also controlled by the right-hand side of \eqref{b-25}. Hence we get \eqref{b-2} and it finishes the proof.
\end{proof}

Now we are going to estimate the nonlinear term of \eqref{auto-1}.
\begin{lemma}
\label{lem-2.3}
Suppose that the assumptions of Lemma \ref{lem-2.2} hold and
$q>\frac{4p}{p-1} $. Then there exists a
positive constant $C=C(N,p,q)>0$ such that
\begin{equation}
\label{b-26}
\begin{split}
\int_{G}\zeta^q\left[\left(\mu_1u^p+\beta
uv\right)^2+\left(\mu_2v^p+\frac{\beta}{2}
u^2\right)^2\right]dxdt&\leq C\int_{G}\zeta^{q-\frac{4p}{p-1}}
\left(|\nabla\zeta|^2+|\Delta\zeta|+|\zeta_t|\right)^{\frac{2p}{p-1}}dxdt.
\end{split}
\end{equation}
\end{lemma}

\begin{proof}
Using the equation of $u$ in \eqref{auto-1} we get
\begin{equation}
\label{b-27}
\begin{split}
\int_{\Omega}\zeta^q\left(\mu_1u^p+\beta
uv\right)^2dxdt&=\int_{\Omega}\zeta^q\left(\mu_1u^p+\beta
uv\right)(u_t-\Delta u)dxdt\\
&=\frac{dg_u}{dt}-K_u+V_u+\mu_1p\int_{\Omega}\zeta^qu^{p-1}|\nabla
u|^2dxdt+\frac{dh}{dt}-H\\
&\quad+\beta\int_{\Omega}\zeta^qv|\nabla
u|^2dxdt+Q+S_u-\frac{\beta}{2}\int_{\Omega}\zeta^qu^2v_tdxdt.
\end{split}
\end{equation}
Analogously, we can get
\begin{equation}
\label{b-28}
\begin{split}
\int_{\Omega}\zeta^q\left(\mu_2v^p+\beta/2
u^2\right)^2dxdt&=\int_{\Omega}\zeta^q\left(\mu_2v^p+\beta/2
u^2\right)(v_t-\Delta v)dxdt\\
&=\frac{dg_v}{dt}-K_v+V_v+\mu_2p\int_{\Omega}\zeta^qv^{p-1}|\nabla
v|^2dxdt+\frac{\beta}{2}\int_{\Omega}\zeta^qu^2v_tdxdt+Q+S_v.
\end{split}
\end{equation}
Following the same argument of deriving \eqref{b-25}, we get from Lemma \ref{lem-2.2} and
\eqref{b-27}-\eqref{b-28} that
\begin{equation}
\label{b-29}
\begin{split}
&\int_{G}\zeta^q\left[\left(\mu_1u^p+\beta
uv\right)^2+\left(\mu_2v^p+\frac{\beta}{2}
u^2\right)^2\right]dxdt\\
&\leq
C\left(\int_{G}\zeta^{q-2}(|\zeta_t|+|\nabla\zeta|^2)(u^{p+1}+v^{p+1})dxdt+\beta\int_{G}\zeta^{q-2}(|\zeta_t|+|\nabla\zeta|^2)u^2vdxdt\right)\\
&\quad+C\int_{G}\zeta^{q-4}(|\Delta\zeta|^2+|\nabla\zeta|^4+|\zeta_t|^2)(u^2+v^2)dxdt\\
&\leq \varepsilon\int_{G}\zeta^q(u^{2p}+v^{2p}+\beta^2u^2v^2)dxdt
+C(\varepsilon)\int_{G}\zeta^{q-\frac{4p}{p-1}}(|\zeta_t|+|\Delta\zeta|+|\nabla\zeta|^2)^{\frac{2p}{p-1}}dxdt,
\end{split}
\end{equation}
where we used the Young's inequality. We notice that the first term on the right-hand side of \eqref{b-29} can be controlled by its left-hand side, then we get \eqref{b-26} and the result is proved.
\end{proof}

Next we prove the Liouville type results for system \eqref{auto-1}

\begin{proof}[Proof of Theorem \ref{th1.1}-(i).]
For any solution $(u,v)$ of \eqref{auto-1} and positive $R$, we define
\begin{equation}
\label{b-30}
(u^R, v^R)=\left(R^{\frac{2}{p-1}}u(Rx, R^2t),
R^{\frac{2}{p-1}}v(Rx, R^2t)\right).
\end{equation}
Then it is easy to see that
\begin{equation}
\label{b-31}
\begin{cases}
\ds u_t^R-\Delta
u^R=\mu_1(u^R)^p+\tilde{\beta}_Ru^Rv^R,\ &x\in\mathbb{R}^{N},\ t\in\mathbb{R},\\
\ds v_t^R-\Delta v^R=\mu_2(v^R)^p+\frac{\tilde{\beta}_R}{2}
(u^R)^2,\ &x\in\mathbb{R}^{N},\ t\in\mathbb{R},
\end{cases}
\end{equation}
where $\tilde{\beta}_R=\beta R^{\frac{2(p-2)}{p-1}}$. By Lemma \ref{lem-2.3} we get that for any $\varepsilon>0$ small, there exists $C(\varepsilon)>0$ such that
\begin{equation}
\label{bj32}
\begin{split}
&\int_{G}\left[\left(\mu_1(u^R)^p+\tilde{\beta}_R
u^Rv^R\right)^2+\left(\mu_2(v^R)^p+\frac{\tilde{\beta}_R}{2}
(u^R)^2\right)^2\right]dxdt\\
&\leq
C\tilde{\beta}_R\int_{G}\zeta^{q-2}(|\zeta_t|+|\nabla\zeta|^2)(u^R)^2v^Rdxdt+
C\int_{G}\zeta^{q-2}(|\zeta_t|+|\nabla\zeta|^2)((u^R)^{p+1}+(v^R)^{p+1})dxdt\\
&\quad+C\int_{G}\zeta^{q-4}\left(|\Delta\zeta|^2+|\nabla\zeta|^4+|\zeta_t|^2\right)((u^R)^2+(v^R)^2)dxdt\\
&\leq\varepsilon\int_{G}\zeta^q((u^R)^{2p}+(v^{2R})^{2p}+(u^R)^2(v^R)^2)dxdt+C(\varepsilon)\int_{G}\zeta^{q-\frac{4p}{p-1}}\left(|\Delta\zeta|+|\nabla\zeta|^2+|\zeta_t|\right)^{\frac{2p}{p-1}}dxdt,
\end{split}
\end{equation}
where $q>\max \frac{4p}{p-1} $, and $C(\varepsilon)$ is independent of $R$. As Lemma \ref{lem-2.3} we see that the first term can be absorbed into the left-hand side, then
\begin{equation}
\label{bj33}
\begin{split}
&\int_{G}\left[\left(\mu_1(u^R)^p+\tilde{\beta}_R
u^Rv^R\right)^2+\left(\mu_2(v^R)^p+\frac{\tilde{\beta}_R}{2}
(u^R)^2\right)^2\right]dxdt\\
&\leq C(\varepsilon)\int_{G}\zeta^{q-\frac{4p}{p-1}}\left(|\Delta\zeta|+|\nabla\zeta|^2+|\zeta_t|\right)^{\frac{2p}{p-1}}dxdt\leq C.
\end{split}
\end{equation}
It implies that
\begin{equation}
\label{b-33}
\begin{split}
&\int_{-R^2}^{R^2}\int_{|y|<R}\left[\left(\mu_1u^p+\beta
uv\right)^2+\left(\mu_2v^p+\frac{\beta}{2} u^2\right)^2\right]dydt\\
&=R^{N+2-\frac{4p}{p-1}}\int_{-1}^{1}\int_{|x|<1}\left[\left(\mu_1(u^R)^p+\tilde{\beta}_R
u^Rv^R\right)^2+\left(\mu_2(v^R)^p+\frac{\tilde{\beta}_R}{2}
(u^R)^2\right)^2\right]dxdt\\
&\leq CR^{N+2-\frac{4p}{p-1}}.
\end{split}
\end{equation}
Since $p<p_B(N)\leq p_S(N)$, by letting $R\to\infty$ we get
\begin{equation}
\int_{-\infty}^\infty\int_{\mathbb{R}^N}
\left[(\mu_1u^p+\beta uv)^2+(\mu_2v^p+\frac{\beta}{2}u^2)^2\right]
dxdt=0.
\end{equation}
It leads to $u=v\equiv 0.$ We finish the proof.
\end{proof}

Next we shall use the Liouville type result to derive the a-priori estimate for the solutions to \eqref{auto-1}. The proof is based on the following doubling lemma, which was introduced in \cite{Souplet-Duke-2007,Serrin-Zou-2002-Acta}.

\begin{lemma}
\label{lem-2.4}
Let $(X,\rho)$ be the complete metric space, and $\emptyset\neq
\Lambda\subset\Sigma\subset X$ with $\Sigma$ closed. Set
$\Gamma=\Sigma\setminus\Lambda$. Let $F:
\Lambda\rightarrow(0,\infty)$ be a map which is bounded on any compact subsets of $\Lambda$, and fix a positive constant $k>0$. If $y\in\Lambda$ satisfies that
\begin{equation}
\label{b-39}
F(y)dist(y, \Gamma)>2k,
\end{equation}
then there exists $x\in\Lambda$ such that
\begin{equation}
\label{b-40}
F(x)dist(x, \Gamma)>2k,\quad F(x)\geq F(y)
\end{equation}
and
\begin{equation}
\label{b-41}
F(z)\leq2 F(x)\quad\forall
z\in\Lambda\cap\overline{B_X\left(x, kF^{-1}(x)\right)}.
\end{equation}
\end{lemma}

Now we are going to give the proof of Theorem \ref{th1.1}-(ii).
\begin{proof}[Proof of Theorem \ref{th1.1}-(ii).]
We prove the conclusion by contradiction. Suppose that the estimate \eqref{a9} fails. Then there exist sequences $D_k=\Omega_k\times(0,T_k)$, $(u_k, v_k)$ and $(y_k, \tau_k)\in D_k$ such that
\begin{equation}
\label{b-42}
M_k(y_k,\tau_k)=(u_k+v_k)^{\frac{p-1}{2}}(y_k,\tau_k)>2kd_P^{-1}((y_k, \tau_k), \partial D_k),
\end{equation}
where $M_k(x)=(u_k+v_k)^{\frac{p-1}{2}}(x)$ and $d_P((x,t),(\tilde x,\tilde t))=|x-\tilde x|+|t-\tilde t|^{1/2}$ denotes the parabolic distance. From Lemma \ref{lem-2.4}, we infer that there exists $(x_k, t_k)\in
D_k$ such that
\begin{equation}
\label{b-43}
\begin{split}
&M_k(x_k, t_k)\geq M_k(y_k, \tau_k),\quad M_k(x_k,
t_k)>2kd_P^{-1}((x_k, t_k), \partial D_k),\\
&M_k(x, t)\leq2M_k(x_k, t_k)\quad \text{whenever}\quad d_P((x,t), (x_k,t_k))\leq kM_k^{-1}(x_k,t_k).
\end{split}
\end{equation}
In the following, we divide our discussion into two cases:
\medskip

\noindent Case 1. If $p=2$, we set
\begin{equation}
\label{b-44}
(\hat{u}_k(y,s),
\hat{v}_k(y,s))=\left(\lambda_k^{2}u_k(x_k+\lambda_ky,
t_k+\lambda_k^2s), \lambda_k^{2}v_k(x_k+\lambda_ky,
t_k+\lambda_k^2s)\right),
\end{equation}
where
\begin{equation}
\label{b-45}
\lambda_k=M_k^{-1}(x_k,t_k)\quad\text{and}\quad (y,
s)\in\hat{D}_k=\{y\in\mathbb{R}^{N}:
|y|<k/2\}\times\left(-\frac{k^2}{4}, \frac{k^2}{4}\right).
\end{equation}
Then we see that $(\hat{u}_k, \hat{v}_k)$ satisfies
\begin{equation}
\label{b-46}
\begin{cases}
\ds \partial_t\hat{u}_k-\Delta
\hat{u}_k=\mu_{1}\hat{u}_k^2+\beta\hat{u}_k\hat{v}_k,\ &(y,s)\in\hat{D}_k,\\
\ds \partial_t\hat{v}_k-\Delta
\hat{v}_k=\mu_{1}\hat{v}_k^2+\frac{\beta}{2}\hat{u}_k^2,\
&(y,s)\in\hat{D}_k,
\end{cases}
\end{equation}
Moreover,
\begin{equation}
\label{b-47}
\hat{u}_k(0,0)+\hat{v}_k(0,0)=\lambda_k^{2}M_k^2=1
\end{equation}
and
\begin{equation}
\label{b-48}
\hat{u}_k(y,s)+\hat{v}_k(y,s)\leq4\quad \text{for}\ (y,
s)\in\hat{D}_k.
\end{equation}
By the standard regularity estimates(see \cite[Theorem
48.1]{Souplet-Book-2019}), we deduce that there exists a subsequence of $(\hat{u}_k, \hat{v}_k)$ converging in
$C_{\mathrm{loc}}^2(\mathbb{R}^{N}\times\mathbb{R})$ to a classical
nonnegative solution $(\hat{u},\hat{v})$ of \eqref{auto-1} with
$p=2$. Furthermore, we infer from \eqref{b-47} that
$\hat{u}(0,0)+\hat{v}(0,0)=1$. Hence $(\hat{u},\hat{v})$ is nontrivial. However, it contradicts Theorem \ref{th1.1}-$(i)$ and the estimate \eqref{a9} is proved in this case.

\noindent Case 2. If $p>2$, the inequality \eqref{b-42} is replaced by
\begin{equation}
\label{b-49}
M_k>2k\left(1+d_P^{-1}((y_k, \tau_k), \partial D_k)\right).
\end{equation}
We set
\begin{equation}
\label{b-50}
(\hat{u}_k(y,s),
\hat{v}_k(y,s))=\left(\lambda_k^{\frac{2}{p-1}}u_k(x_k+\lambda_ky,
t_k+\lambda_k^2s), \lambda_k^{\frac{2}{p-1}}v_k(x_k+\lambda_ky,
t_k+\lambda_k^2s)\right),
\end{equation}
where $\lambda_k=M_k^{-1}$ as before and $(y,
s)\in\hat{D}_k=\{y\in\mathbb{R}^{N}:
|y|<k/2\}\times\left(-\frac{k^2}{4}, \frac{k^2}{4}\right)$. A direct
computation shows that $(\hat{u}_k,\hat{v}_k)$ satisfies
\eqref{b-31} with $\tilde{\beta}_R=\beta
\lambda_k^{\frac{2(p-2)}{p-1}}$, and $(\hat{u}_k,\hat{v}_k)$ verifies \eqref{b-47} and
\eqref{b-48}. Since $\lambda_k\rightarrow0$ as
$k\rightarrow\infty$, we get that some subsequence of $(\hat{u}_k,
\hat{v}_k)$ converges in $C_{\mathrm{loc}}^2(\mathbb{R}^{N}\times\mathbb{R})$
to a classical nonnegative solution $(\hat{u},\hat{v})$ of
\eqref{auto-1} with $\beta=0$. This contradicts \cite[Theorem A]{Souplet-IUMJ-2007} and we finish the whole proof.
\end{proof}

\subsection{The Liouville results in half space}

In this subsection, we focus on the Liouville type results of \eqref{auto-10} in half-space $\mathbb{R}_+^N$, i.e., we shall prove Theorem \ref{th1.2}. To achieve this goal we need the following monotonicity result, which is due to \cite[Theorem 2.3]{Souplet-2016-Math-Ann}.

\begin{lemma}
\label{lem-3.1}
Let $N\geq1$ and $k\geq2$ and consider the following system
\begin{equation}
\label{b-52}
\begin{cases}
\frac{\partial u_i}{\partial t}-\Delta u_i=f_i(u_1,\cdot\cdot\cdot,u_k),\ (x, t)\in\mathbb{R}_{+}^N\times\mathbb{R},\
&i=1,\cdot\cdot\cdot, k,\\
u_i=0\ (x, t)\in\partial\mathbb{R}_{+}^N\times\mathbb{R},\
&i=1,\cdot\cdot\cdot, k.
\end{cases}
\end{equation}
Suppose that $f_i: [0,\infty)^k\rightarrow\mathbb{R}$ are
$C^1$-functions satisfying
\begin{itemize}
\item [($L_1$)] $f_i(0,\cdot\cdot\cdot,0)=0$ and $\sum_{j=1}^k\frac{\partial f_i}{\partial u_j}(0,\cdot\cdot\cdot,0)\leq0$ for all $i$;

\item [($L_2$)] $\frac{\partial f_i}{\partial u_j}(u_1,\cdot\cdot\cdot,u_k)\geq0$ for all $(u_1,\cdot\cdot\cdot,u_k)\in[0,\infty)^k$ and all $i\neq j$ ~($i, j=1,\cdot\cdot\cdot,k$);

\item [($L_3$)] any nontrivial nonnegative bounded solution of
\eqref{b-52} is positive in $\mathbb{R}_{+}^N\times\mathbb{R}$.
\end{itemize}
Then any nontrivial nonnegative bounded solution of \eqref{b-52} is increasing in $x_1$:
\begin{equation}
\label{b-53}
\frac{\partial u_i}{\partial x_1}(x,t)>0,\quad\forall
(x,t)\in\mathbb{R}_+^N\times \mathbb{R},\ i=1, \cdot\cdot\cdot,k.
\end{equation}
\end{lemma}

Based on the above lemma, we derive the following result.

\begin{lemma}
\label{lem-3.2}
Assume that $1<p<\frac{N(N+2)}{(N-1)^2}(N\leq5)$. Then any
nontrivial nonnegative bounded solution of \eqref{auto-10} is increasing in $x_1$:
\begin{equation}
\label{b-54}
\frac{\partial u}{\partial x_1}(x,t)>0\quad\text{and}\quad
\frac{\partial v}{\partial x_1}(x,t)>0\quad\forall
(x,t)\in\mathbb{R}_+^N\times \mathbb{R}.
\end{equation}
\end{lemma}

\begin{proof}
Let $(u, v)$ be any nontrivial nonnegative bounded solution of
\eqref{auto-10}. We shall apply Lemma \ref{lem-3.1} to prove the conclusion. Let
\begin{equation}
\label{b-55}
f_1(u,v)=\mu_1u^p+\beta uv\quad \text{and}\quad
f_2(u,v)=\mu_2v^p+\frac{\beta}{2}u^2.
\end{equation}
Since $\mu_1, \mu_2, \beta>0$, we could easily check that
$(L_1)$-$(L_2)$ of Lemma \ref{lem-3.1} hold. It remains to verify the condition of Lemma \ref{lem-3.1} $(L_3)$. We claim that both $u$ and $v$ are either identically zero or positive
in $\mathbb{R}_+^{N}\times\mathbb{R}$. By strong maximum principle (see \cite[Proposition 6.1]{Souplet-2016-Math-Ann} or \cite[Proposition 52.21]{Souplet-Book-2019}) we get that either $(u,v)$ is positive in $\mathbb{R}_+^{N}\times\mathbb{R}$ or there exists $t_0\in\mathbb{R}$ such that $(u,v)=0$ in $\mathbb{R}_+^{N}\times(-\infty, t_0]$. In the latter case, since $(u,v)$ is bounded, it follows that $(u,v)$ satisfying
\begin{equation}
\label{b-56}
\begin{cases}
u_t-\Delta u\leq a_{11}u+a_{12}v,\ (x,t)\in
\mathbb{R}_+^{N}\times\mathbb{R},\\
v_t-\Delta v\leq a_{21}v+a_{22}u,\ (x,t)\in
\mathbb{R}_+^{N}\times\mathbb{R},
\end{cases}
\end{equation}
where $a_{ij}>0(i,j=1,2)$ are constants. Then by using the maximum principle again, we obtain that $u=v\equiv0$ and it proves the claim. Now we conclude that the condition $(L_3)$ holds. Indeed, if $v=0$, then we get from system \eqref{auto-10} that $u=0$. Therefore $(u,v)$ is the trivial solution and contradiction arises. If $u=0$ and $v\neq0$, then $v$ is a positive solution of $v_t-\Delta v=\mu_1v^p(1<p<\frac{N(N+2)}{(N-1)^2})$ and it contradicts  \cite[Theorem 1]{Bidaut-1998-Initial}. Hence, both $(u,v)$ are strictly positive and we finish the whole proof.
\end{proof}

In the end we prove Theorem \ref{th1.2}.

\begin{proof}[Proof of Theorem \ref{th1.2}.]
From Lemma \ref{lem-3.2}, we know that \eqref{b-54} holds. For each $R>0$, we define
\begin{equation}
\label{b-57}
u_R(x_1, x', t)=u(x_1+R, x', t)\quad\text{and}\quad v_R(x_1, x',
t)=v(x_1+R, x', t)
\end{equation}
for $(x_1, x', t)\in(-R,\infty)\times\mathbb{R}^{N-1}\times\mathbb{R}$. From the boundedness of $(u, v)$ and classical parabolic estimates, sending $R\rightarrow\infty$, we get that $(u_R, v_R)$ converges uniformly to $(u_\infty, v_\infty)$ in any compact set after passing to a subsequence if necessary, where $(u_\infty, v_\infty)$ is a bounded, nonnegative classical solution of problem \eqref{auto-1} in $\mathbb{R}^{N-1}\times\mathbb{R}$. The monotonicity of $(u, v)$ implies that $(u_\infty, v_\infty)$ is positive and independent of $x_1$. Thus we obtain a bounded, positive classical solution of problem \eqref{auto-1} in $\mathbb{R}^{N-1}\times\mathbb{R}$. This contradicts Theorem \ref{th1.1} when $N\geq2$. For the case $N=1$ it reduces to an ODE system for which the nonexistence is obvious.
\end{proof}

\bigskip
\section{Existence of periodic solutions}
\subsection{Positive coefficient case}
In this subsection we study the existence of periodic solutions of the two coupled system \eqref{auto-11}. Before giving the proof,
we make the following preparations. We set $X=BUC(\Omega\times(0,T))$ to be the space of bounded uniformly continuous functions, which is equipped with the $L^\infty$-norm $\|\cdot\|_\infty$. For each $z$,
we denote $z^+(x,t)=\max\{z(x,t),0\}$. By a slightly abuse of notation, by $\|\cdot\|_\infty$ we denote both the norm in
$L^\infty(\Omega\times(0,T))$ and $L^\infty(\Omega)$.

In our discussion, we need the following results on the (possibly) sign-changing solution and eigenvalue problem (see \cite[Propositions 10-11]{Quittner-JDE-2016} or
\cite{Esteban-TAMS-1986,Beltramo-CPDE-1984}).

\begin{lemma}
\label{lem3.1}
Suppose that $\Omega\subset\mathbb{R}^{N}$ is a $C^3$-smooth bounded domain, $T>0$ and $f\in X$. Then we have the following conclusions.
\begin{itemize}
\item [$(i)$] The scalar periodic problem
\begin{equation}
\label{c1}
\begin{cases}
w_t-\Delta w=f,\quad &(x,t)\in\Omega\times(0,T)\\
w=0, &(x,t)\in\partial\Omega\times(0,T)\\
w(\cdot,0)=w(\cdot,T), &x\in\Omega,
\end{cases}
\end{equation}
has a unique solution $w$. Moreover, the mapping $\mathcal {T}: f\in X\rightarrow w^+\in X$ is compact.

\item [$(ii)$] There exists $\lambda_1^T>0$ such that the problem
\begin{equation}
\label{c2}
\begin{cases}
-\psi_t-\Delta\psi=\lambda_1^T\psi,\quad &(x,t)\in\Omega\times(0,T)\\
\psi=0, &(x,t)\in\partial\Omega\times(0,T)\\
\psi(\cdot,0)=\psi(\cdot,T), &x\in\Omega,
\end{cases}
\end{equation}
possesses a positive solution $\psi$.
\end{itemize}
\end{lemma}

To solve the problem, we utilize the homotopy deformation to decouple the system and apply the degree theory to show the existence of a solution. The homotopy transformation is given below
\begin{equation}
\label{c3}
\begin{cases}
\ds u_t-\Delta u=\theta(\mu_1u^p+\beta uv)+(1-\theta)(\lambda u+u^2),\ &x\in\Omega,\ t\in(0, \infty),\\
\ds v_t-\Delta v=\theta(\mu_2v^p+\frac{\beta}{2} u^2)+(1-\theta)(\lambda v+v^2),\ &x\in\Omega,\
t\in(0, \infty),\\
u(x,t)=v(x,t)=0,\quad &\text{on}\ \partial\Omega\times(0, \infty),
\end{cases}
\end{equation}
where $\theta\in[0,1]$ and $\lambda>0$ is a constant. Our next issue is to establish a-priori bounds for system \eqref{c3}. To achieve this goal we should prove a Liouville type result for the following system
\begin{equation}
\label{c4}
\begin{cases}
\ds u_t-\Delta
u=\theta\left(\mu_1u^2+\beta uv\right)+(1-\theta)v^2,\ &x\in\Omega,\ t\in\mathbb{R},\\
\ds v_t-\Delta v=\theta\left(\mu_2v^2+\frac{\beta}{2}
u^2\right)+(1-\theta)u^2,\ &x\in\Omega,\ t\in\mathbb{R},
\end{cases}
\end{equation}
where $\mu_1, \mu_2, \beta>0$ and $\theta\in[0,1]$. Following almost the same arguments as in Theorem \ref{th1.1}-$(i)$, we obtain the following result.

\begin{lemma}
\label{lem3.2}
Suppose that $N\leq5$, $\mu_1, \mu_2, \beta>0, \theta\in[0,1]$ and $\Omega=\mathbb{R}^{N}$. Then the system \eqref{c4} has no nontrivial nonnegative classical solution.
\end{lemma}

With Lemma \ref{lem3.2} we prove that any periodic solution of  \eqref{c3} is uniformly bounded for any $\theta\in[0,1]$.

\begin{lemma}
\label{lem3.3}
Suppose that $N\leq5$, $\mu_1, \mu_2, \beta>0, \theta\in[0,1]$ and $2\leq p<p_B(N)$. Then there exists $C>0$ such that any positive $T$-periodic solution $(u,v)$ of \eqref{c3} satisfies
\begin{equation}
\label{c5}
u(x,t)+v(x,t)\leq C,\quad\forall
(x,t)\in\Omega\times(0,\infty).
\end{equation}
\end{lemma}

\begin{proof}
The idea follows the proof of Theorem \ref{th1.1}-(ii). We prove the result by contradiction. If \eqref{c5} is not satisfied, then there exist sequences $D_k:=\Omega_k \times(0, T_k), (u_k, v_k)$ and $(y_k, \tau_k)\in D_k$ such that
\begin{equation}
\label{c6}
M_k(x,t)\mid_{x=y_k,t=\tau_k}=(u_k+v_k)^{\frac{p-1}{2}}(y_k,\tau_k)>2k.
\end{equation}
From the doubling lemma (Lemma \ref{lem-2.4}), we infer that there exists $(x_k, t_k)\in D_k$ such that
\begin{equation}
\label{c7}
\begin{split}
&M_k(x_k, t_k)\geq M_k(y_k, \tau_k),\quad M_k(x_k,t_k)>2k,\\
&M_k(x, t)\leq2M_k(x_k, t_k)\quad \text{whenever}\quad d_P((y_k,
    \tau_k), \partial D_k)\leq kM_k^{-1}.
\end{split}
\end{equation}
We set
\begin{equation}
(\hat{u}_k(y,s),\hat{v}_k(y,s))=\left(\lambda_k^{\frac{2}{p-1}}u_k(x_k+\lambda_ky,t_k+\lambda_k^2s),\lambda_k^{\frac{2}{p-1}}v_k(x_k+\lambda_ky,t_k+\lambda_k^2s)\right),
\end{equation}
where $(y, s)\in\hat{D}_k=\{y\in\mathbb{R}^{N}: |y|<k/2\}\times\left(-\frac{k^2}{4}, \frac{k^2}{4}\right)$ and $\lambda_k=M_k^-1(x_k,t_k)$, which tends to $0$ as $k\to\infty$. Then it follows that $(\hat{u}_k(y,s), \hat{v}_k(y,s))$ satisfies
\begin{equation}
\label{c--9}
\begin{cases}
\ds \hat{u}_{k,t}-\Delta
\hat{u}_k=\theta\left(\mu_1\hat{u}_k^p+\beta\lambda_k^{\frac{2(p-2)}{p-1}}\hat{u}_k\hat{v}_k\right)+(1-\theta)\left(\lambda\lambda_k^{2}\hat{u}_k+\lambda_k^{\frac{2(p-2)}{p-1}}\hat{u}_k^2\right),\\
\ds\hat{v}_{k,t}-\Delta\hat{v}_k=\theta\left(\mu_2\hat{v}_k^p+\frac{\beta}{2}\lambda_k^{\frac{2(p-2)}{p-1}}\hat{u}_k^2\right)+(1-\theta)\left(\lambda\lambda_k^{2}\hat{v}_k+\lambda_k^{\frac{2(p-2)}{p-1}}\hat{v}_k^2\right).
\end{cases}
\end{equation}
If $p>2$, we get that $(\hat{u}_k,\hat{v}_k)$ converges (after passing a subsequence if necessary) in $C_{\mathrm{loc}}^2(\mathbb{R}^{N}\times\mathbb{R})$
to a classical nonnegative solution $(\hat{u},\hat{v})$ of
\eqref{auto-1} with $\beta=0$. This contradicts \cite[Theorem
    A]{Souplet-IUMJ-2007}. If $p=2$, we know that some
subsequence of $(\hat{u}_k, \hat{v}_k)$ converges in $C_{\mathrm{loc}}^2(\mathbb{R}^{N}\times\mathbb{R})$ to a classical nonnegative solution $(\hat{u},\hat{v})$ of \eqref{c4}. This is a contradiction.
\end{proof}

Now we are ready to give the proof of Theorem \ref{th1.3}.
\begin{proof}[Proof of Theorem \ref{th1.3}.]
Let $\mathcal{T}$ be the compact mapping by Lemma \ref{lem3.1}. We define the operator $\mathcal{L}$: $X\times X\to X\times X$ by
\begin{equation}
\label{c13}
\mathcal {L}(u,v)=\left(\mathcal {T}(\mu_1u^p+\beta uv), \mathcal
{T}(\mu_2v^p+\frac{\beta}{2}u^2)\right).
\end{equation}
It is clear that $\mathcal{L}$ is compact and the nontrivial fixed point of \eqref{c13} corresponds to a non-negative periodic solutions of \eqref{auto-11}. Indeed, if $(u,v)\neq(0,0)$ is a fixed point of $\mathcal {T}$, then we get that $(u,v)=(w^+, z^+)$ and $(w, z)$ is
$T$-periodic solution of
\begin{equation}
\label{c14}
\begin{cases}
\ds w_t-\Delta
w=\mu_1(w^+)^p+\beta w^+z^+,\ &x\in\Omega,\ t\in(0, \infty),\\
\ds z_t-\Delta z=\mu_2(z^+)^p+\frac{\beta}{2}(w^+)^2,\ &x\in\Omega,\
t\in(0, \infty),\\
w(x,t)=z(x,t),\quad &\text{on}\ \partial\Omega\times(0, \infty).
\end{cases}
\end{equation}
Using the maximum principle $w, z\geq0$. Furthermore, we claim that both $w$ and $z$ are strictly positive. In fact, if $z\equiv0$, from the second equation of \eqref{c14} we infer that $w\equiv0$. This contradicts $(w,z)\neq(0,0)$. On the other hand, if $w\equiv0$, then $z$ satisfies $z_t-\Delta z=\mu_2z^p$. Following a similar argument as we did in Lemma \ref{lem3.3}, we get that any positive periodic solution of $z$ of $z_t-\Delta z=\mu_2z^p$ is bounded. That is, $z(x,t)\leq C$, where $C>0$. Multiplying the equation by $\varphi_1$ and integrating over $\Omega$ yields
\begin{equation*}
\frac{d}{dt}\int_{\Omega}z(\cdot,t)\varphi_1dx=\int_{\Omega}\left(-\lambda_1+\mu_2z^{p-1}\right)\varphi_1zdx<0,
\end{equation*}
provided $\mu_2\geq0$ sufficiently small. It implies that the function $t\mapsto\int_{\Omega}z(\cdot,t)\varphi_1dx$ is strictly decreasing in time, which contradicts the periodicity of $z$.

Finally, it remains to show the existence of a nontrivial fixed point of $\mathcal{L}$. We shall apply the techniques of \cite{Quittner-JDE-2016} to compute the Leray-Schauder degree of
$F(r)=deg(I-\mathcal {T}, B_r, 0)$ for small and large $r>0$ respectively. Here $I$ denotes the identity map and $B_r$ is the ball in $X\times X$ with radius $r$ centered at zero. We first prove that $F(r)=1$ for $r>0$ sufficiently small. In order to prove it, we define the homotopy
\begin{equation}
\label{c15}
\mathcal {L}_s(u,v)=\left(\mathcal {T}(s(\mu_1u^p+\beta uv)),
\mathcal {T}(s(\mu_2v^p+\frac{\beta}{2}u^2))\right),\quad s\in[0,1].
\end{equation}
First, we shall show that $\mathcal {L}_s$ is admissible by contradiction arguments. Assume that there exists a nontrivial fixed point $(u,v)$ of $\mathcal {L}_h$ satisfying $\|(u,v)\|_\infty=r$ with $r$ sufficiently small. Fixing a positive number $t$ such that $\|(u(\cdot,t),v(\cdot,t))\|_\infty=r$. Without loss of generality
we may assume $\|u(\cdot,t)\|_\infty=r$. Since $(u, v)$ is a
positive periodic solution of \eqref{auto-11} with the right-hand
sides multiplied by $s$, it follows from the variation-of-constants
formula that
\begin{equation}
\label{c16}
\begin{split}
r=\|u(\cdot,t+T)\|_\infty&\leq
e^{-\lambda_1T}\|u(\cdot,t)\|_\infty+CTs\left(\|u\|_\infty^p+\|u\|_\infty\|v\|_\infty\right)\\
&\leq e^{-\lambda_1T}r+C(r^p+r^2).
\end{split}
\end{equation}
This is a contradiction for $r$ sufficiently small. Hence we obtain
that $F(r)=1$ for $r>0$ sufficiently small. On the other hand, we
consider the case $r>0$ large. Let us set
\begin{equation}
\label{c17}
\begin{split}
\mathcal {L}_s(u,v)=\left(\mathcal {T}(s(\mu_1u^p+\beta uv)+(1-s)(\lambda u+u^2)),
\mathcal {T}(s(\mu_2v^p+\frac{\beta}{2}u^2)+(1-s)(\lambda
v+v^2))\right),
\end{split}
\end{equation}
where $s\in[0,1]$ and $\lambda=\lambda_1^T+2$. The estimates in
Lemma \ref{lem3.3} guarantee that this homotopy is admissible if $r$ is large enough. Hence it is sufficient to show that problem
\eqref{c17} does not possess positive periodic solutions if $s=0$.
We use the contradiction arguments. Assume that $(u,v)$ is a
positive $T$-periodic solution of the system
\begin{equation}
\label{c18}
\begin{cases}
\ds u_t-\Delta
u=\lambda u+u^2,\ &x\in\Omega,\ t\in(0, \infty),\\
\ds v_t-\Delta v=\lambda v+v^2,\ &x\in\Omega,\
t\in(0, \infty),\\
u(x,t)=v(x,t)=0,\quad &\text{on}\ \partial\Omega\times(0, \infty).
\end{cases}
\end{equation}
Multiplying the first equation by the eigenfunction $\psi$ from
Lemma \ref{lem3.1}, integrating over $\Omega\times (0, T)$ and using integration by parts we obtain
\begin{equation}
\label{c19}
\lambda_1^T\int_{0}^T\int_{\Omega}u\psi dxdt\geq\lambda\int_{\Omega}u\psi dxdt.
\end{equation}
This is a contradiction.
\end{proof}

\subsection{Negative coefficients case}
In this subsection we prove the Liouville type results and derive the existence of periodic solutions for the two coupled system \eqref{auto-1} with negative coefficients. Let $f(u,v)=-|\mu_1|u^2+\beta uv$ and $g(u,v)=-|\mu_2|v^2+\beta/2 u^2$. We first prove the Liouville type results.

\begin{proof}[Proof of Theorem \ref{th1.4}.]
We first claim that there exists a real number $K>0$ such that
\begin{equation}
\label{c20}
(Kg-f)(u-Kv)\geq C(u+Kv)(u-Kv)^2,\ \text{where}\ C>0.
\end{equation}
Indeed, by direct computations we see that
\begin{equation}
\label{c21}
\begin{split}
(Kg-f)&=-K|\mu_2|v^2+\frac{\beta}{2}Ku^2+|\mu_1|u^2-\beta uv\\
&=v^2(|\mu_1|\theta^2+\frac{\beta}{2}K\theta^2-\beta\theta
-K|\mu_2|):=v^2h_K(\theta),
\end{split}
\end{equation}
where $\theta=u/v$. We insert $\theta=K$ into $h_K(\theta)$ and $h_K(K)$ can be regarded as polynomial of degree 3. It is clear that $h_K(K)=0$ has only one positive root
$$K_+=\frac{-|\mu_1|+\sqrt{|\mu_1|^2+2\beta(\beta+|\mu_2|)}}{\beta}>0.$$
Concerning $K_+$, we claim that
\begin{equation}
\label{c22}
m(\theta)=\frac{h_{K_+}(\theta)}{\theta^2-K_+^2}\geq C>0.
\end{equation}
In fact, since $m(\theta)>0$ in $(0,K_+)\cap(K_+,\infty)$ and $m(\theta)$ has positive limit as $\theta$ goes to $K_+$ or $+\infty$, it follows that the claim \eqref{c23} holds. We infer from the following basic inequality
\begin{equation}
\label{c23}
(x^k-y^k)(x-y)\geq C_k(x+y)^{k-1}(x-y)^2
\end{equation}
that
\begin{equation}
\label{c24}
(K_+g-f)(u-K_+v)=v^3h_{K_+}(\theta)(\theta-K_+)\geq C(u+K_+v)(u-K_+v)^2.
\end{equation}
Let $w=u-K_+v$. We infer from \eqref{c20} that
\begin{equation}
(w_t-\Delta w)sign(w)\leq -C|u-Kv|^2
\end{equation}
where $C>0$ is a constant. Then it follows from \cite[Proposition
4]{Quittner-JDE-2016} that $w\equiv0$ on $X$. That is, $u=K_+v$
satisfies
\begin{equation}
\label{c25}
u_t-\Delta
u=\left(\sqrt{|\mu_1|^2+2\beta(\beta+|\mu_2|)}-2|\mu_1|\right)u^2,\
(x,t)\in X\times\mathbb{R}.
\end{equation}
If $X=\mathbb{R}^{N}$, we claim that $u+v$ is bounded. Indeed, suppose $u(y_k,\tau_k)+v(y_k,\tau_k)\rightarrow\infty$ for some $(y_k,\tau_k)\in\mathbb{R}^{N}\times\mathbb{R}$ as $k\rightarrow\infty$. Set $M=\sqrt{u+v}$. From Lemma \ref{lem-2.4} we derive that there exists $(x_k, t_k)$ such that \eqref{c7} holds. We define
\begin{equation}
\label{c27}
(\hat{u}_k(y,s),
\hat{v}_k(y,s))=\left(\lambda_k^{2}u_k(x_k+\lambda_ky,
t_k+\lambda_k^2s), \lambda_k^{2}v_k(x_k+\lambda_ky,
t_k+\lambda_k^2s)\right),
\end{equation}
where $(y, s)\in\hat{D}_k=\{y\in\mathbb{R}^{N}:
|y|<k/2\}\times\left(-\frac{k^2}{4}, \frac{k^2}{4}\right)$. As in
the proof of Lemma \ref{lem3.3}, we get that there exists some subsequence of
$(\hat{u}_k, \hat{v}_k)$ converging in
$C_{\mathrm{loc}}^2(\mathbb{R}^{N}\times\mathbb{R})$ to a classical
non-negative bounded solution of \eqref{auto-1} with $p=1$. Therefore, we can assume that $(u,v)$ is bounded. Using Theorem \ref{th1.1}-$(i)$ we get that $u\equiv0$ and it implies that $u=v\equiv0$. Such conclusion also holds if $X=\mathbb{R}_+^{N}$. Hence, we finish the proof.
\end{proof}

Next we prove Theorem \ref{th1.4}-$(ii)$. By using similar
arguments as in Lemma \ref{lem3.3}, we have the following result.

\begin{lemma}
\label{lem3.4}
Assume that $N\leq 5$, $\mu_1, \mu_2<0, \beta>0$ and
$p=2$. Then for any positive $T$-periodic solution $(u,v)$ of \eqref{c3} there exists a generic constant $C>0$ (independent of $\theta$ such that
\begin{equation}
\label{c28}
u(x,t)+v(x,t)\leq C\quad\mbox{for  all}\quad (x,t)\in\Omega\times(0,\infty).
\end{equation}
\end{lemma}

\begin{proof}[Proof of Theorem \ref{th1.4}-$(ii)$]
By arguing as in Theorem \ref{th1.3} we let $\mathcal {T}$ be the compact operator defined in Lemma \ref{lem3.1}. We set $\mathcal{L}$ to be the operator:
\begin{equation}\label{c29}
\mathcal {L}(u,v)=\left(\mathcal {T}(-|\mu_1|u^p+\beta uv), \mathcal{T}(-|\mu_2|v^p+\frac{\beta}{2}u^2)\right),
\end{equation}
which is compact from $X\times X$ to itself. It is clear that nontrivial fixed points of \eqref{c29} correspondes to the nonnegative periodic solutions of \eqref{auto-11}. Indeed, if $(u,v)\neq(0,0)$ is a fixed point of $\mathcal {T}$, then we see $(u,v)=(w^+, z^+)$ and $(w, z)$ is $T$-periodic solution of
\begin{equation}
\label{c30}
\begin{cases}
\ds w_t-\Delta w=-|\mu_1|(w^+)^p+\beta w^+z^+,\ &x\in\Omega,\ t\in(0, \infty),\\
\ds z_t-\Delta z=-|\mu_2|(z^+)^p+\frac{\beta}{2}(w^+)^2,\
&x\in\Omega,\ t\in(0, \infty),\\
w(x,t)=z(x,t),\quad &\text{on}\ \partial\Omega\times(0, \infty).
\end{cases}
\end{equation}
By maximum principle we see that $w, z\geq0$. In addition, we claim that $w>0$ and $z>0$. In fact, if $z\equiv0$, we get that $w\equiv0$ from the second equation of \eqref{c30}. This
contradicts $(w,z)\neq(0,0)$. On the other hand, if $w\equiv0$, then $z$ satisfies
$$z_t-\Delta z=-|\mu_2|z^p.$$
Multiplying the above equation by $\varphi_1$ and integrating over $\Omega$ we have
\begin{equation*}
\frac{d}{dt}\int_{\Omega}z(\cdot,t)\varphi_1dx=\int_{\Omega}\left(-\lambda_1-|\mu_2|z^{p-1}\right)\varphi_1zdx<0.
\end{equation*}
It implies that the function $t\mapsto\int_{\Omega}z(\cdot,t)\varphi_1dx$ is time decreasing. This contradicts the periodicity of $z$. Hence both $u$ and $v$ are strictly positive.

In the end, repeating the arguments of Theorem \ref{th1.3}, we know that the Leray-Schauder degree of $F(r)=deg(I-\mathcal {T}, B_r, 0)=1$ for $r>0$ small and $F(r)=0$ for large $r>0$, which confirms the existence of a solution of \eqref{auto-11} and it finishes the proof.
\end{proof}

\bigskip
\section{Three coupled system}
\subsection{Liouville type results}
In this subsection we derive Liouville type results together with singularity and decay estimates for the three coupled system \eqref{auto-2}. As for Theorem \ref{th1.1} we apply Lemma \ref{lem-2.1} to derive the following interior estimates for \eqref{auto-2}.

\begin{lemma}
\label{lem-4.1}
Assume that $\mu_1, \mu_2, \mu_3, \beta>0$ and
$1<p<\frac{N(N+2)}{(N-1)^2}$. Let $\zeta\in\mathcal {D}(G)$ with
value in $[0,1]$ and $q>4$. Then there exists a constant $C=C(N,p,q)>0$ such that
\begin{equation}
\label{d-1}
\begin{split}
&\int_{G}\zeta^{q}\left(u^{-2}|\nabla u|^4+v^{-2}|\nabla v|^4+
w^{-2}|\nabla w|^4\right)dxdt+\int_{G}\zeta^{q}
\left(u^{p-1}|\nabla u|^2+v^{p-1}|\nabla v|^2+w^{p-1}|\nabla w|^2\right)dxdt\\
&\quad+\int_{G}\zeta^q(u_t^2+v_t^2+w_t^2)dxdt+\int_{G}\zeta^{q}\left(|\nabla u|^2vwu^{-1}+|\nabla v|^2uwv^{-1}+|\nabla w|^2uvw^{-1}\right)dxdt\\&\leq C\int_{G}\zeta^{q-2}\left(|\zeta_t|+|\nabla\zeta|^2\right)\left(u^{p+1}+v^{p+1}+w^{p+1}\right)dxdt+C\beta\int_G\zeta^{q-2}\left(|\zeta_t|+|\nabla\zeta|^2\right)uvwdxdt\\
&\quad+C\int_{G}\zeta^{q-4}\left(|\Delta\zeta|^2+|\nabla\zeta|^4+|\zeta_t|^2\right)\left(u^2+v^2+w^2\right)dxdt.
\end{split}
\end{equation}
\end{lemma}

\begin{proof}
We shall follow the arguments as in Lemma \ref{lem-2.2}. First, we apply Lemma \ref{lem-2.1} for positive solutions $(u,v,w)$ of \eqref{auto-2}. To simplify our notation, we set $(u_1, u_2, u_3)=(u,v,w)$ and define
$\xi=\zeta^q=\zeta^q(t,\cdot)$ for any $t\in[t_1, t_2]$, with $m=0$ and any real $d\neq2$
\begin{equation}
\label{d-2}
\begin{split}
&\frac{\left(2Nd-(N-1)d^2\right)}{4N}\int_{\Omega}\zeta^qu_i^{-2}|\nabla u_i|^4dx-\frac{N-1}{N}\int_{\Omega}\zeta^q(\Delta u_i)^2dx -\frac{(N+2)d}{2N}\int_{\Omega}\zeta^qu_i^{-1}|\nabla u_i|^2\Delta u_idx\\
&\leq\frac{d}{2}\int_{\Omega}u_i^{-1}|\nabla u_i|^2\nabla u_i\nabla(\zeta^q)dx+\int_{\Omega}\Delta u_i\nabla u_i\nabla(\zeta^q)dx+\frac{1}{2}\int_{\Omega}|\nabla u_i|^2\Delta(\zeta^q)dx,\quad i=1,2,3.
\end{split}
\end{equation}
Next we shall give the estimates for \eqref{d-2} term by term. A direct computation shows that
\begin{equation}
\label{d-3}
\begin{split}
-\int_{\Omega}\zeta^q(\Delta u_i)^2dx&=-\int_{\Omega}\zeta^q\Delta
u_i(u_{i,t}-\mu_iu_i^p-\beta u_jv_k)dx\\
&=\frac{df_{u_i}}{dt}-P_{u_i}+\int_{\Omega}\nabla
u_i\nabla(\zeta^q)u_{i,t}dx-\mu_ip\int_{\Omega}\zeta^qu_i^{p-1}|\nabla u_i|^2dx-\mu_i\int_{\Omega}\nabla(\zeta^q)\nabla u_iu_i^{p}dx\\
&\quad-\beta\int_{\Omega}\nabla(\zeta^q)\nabla u_iu_ju_kdx-\beta\int_{\Omega}\zeta^q\nabla u_i\nabla
u_ju_kdx-\beta\int_{\Omega}\zeta^q\nabla u_i\nabla u_ku_jdx,
\end{split}
\end{equation}
and
\begin{equation}
\label{d-4}
\begin{split}
-\int_{\Omega}\zeta^qu_i^{-1}|\nabla u_i|^2\Delta u_idx&=-\int_{\Omega}\zeta^qu_i^{-1}|\nabla u_i|^2(u_{i,t}-\mu_iu_i^p-\beta u_ju_k)dx\\
&=-\int_{\Omega}\zeta^qu_i^{-1}|\nabla u_i|^2u_{i,t}dx+\mu_i\int_{\Omega}\zeta^qu_i^{p-1}|\nabla u_i|^2dx
+\beta\int_{\Omega}\zeta^q|\nabla u_i|^2u_i^{-1}u_ju_kdx,
\end{split}
\end{equation}
and
\begin{equation}
\label{d-5}
\begin{split}
\int_{\Omega}\Delta u_i\nabla u_i\nabla(\zeta^q)dx&=\int_{\Omega}\nabla u_i\nabla(\zeta^q)(u_{i,t}-\mu_iu_i^p-\beta u_ku_j)dx\\
&=\int_{\Omega}\nabla u_i\nabla(\zeta^q)u_{i,t}dx-\mu_i\int_{\Omega}\nabla u_i\nabla(\zeta^q)u_i^pdx-\beta\int_{\Omega}\nabla
u_i\nabla(\zeta^q)u_ku_jdx
\end{split}
\end{equation}
where $i=1,2,3$ ($i, j, k\in\{1, 2, 3\}$ are pairwise different), and
\begin{equation}
\label{d-6}
f_{u_i}=\frac{1}{2}\int_{\Omega}\zeta^q|\nabla u_i|^2dx\quad\text{and}\quad P_{u_i}=\frac{1}{2}\int_{\Omega}\left(\zeta^q\right)_t|\nabla
u_i|^2dx.
\end{equation}

Substituting \eqref{d-3}-\eqref{d-5} into \eqref{d-2}, we get
\begin{equation}
\label{d-7}
\begin{split}
&A_0\int_{\Omega}\zeta^qu_i^{-2}|\nabla u_i|^4dx+B_i\int_{\Omega}\zeta^qu_i^{p-1}|\nabla u_i|^2dx+C_1\int_{\Omega}\zeta^q|\nabla u_i|^2u_i^{-1}u_ju_kdx-\frac{N-1}{N}Q_{u_i}\\
&\leq-\frac{N-1}{N}\left(\frac{df_{u_i}}{dt}-P_{u_i}\right)+\frac{(N+2)d}{2N}Y_{u_i}+\frac{1}{N}Z_{u_i}-\frac{1}{N}V_{u_i}+\frac{d}{2}W_{u_i}+\frac{1}{2}R_{u_i}-\frac1NS_{u_i},
\end{split}
\end{equation}
where
\begin{equation}
\label{d-8}
\begin{split}
&A_0=\frac{2Nd-(N-1)d^2}{4N},\ B_i=\frac{(N+2)d-2(N-1)p}{2N}\mu_i,\
C_1=\frac{(N+2)d}{2N}\beta,\\
&Y_{u_i}=\int_{\Omega}\zeta^qu_{i}^{-1}|\nabla u_i|^2u_{i,t}dx,\
Z_{u_i}=\int_{\Omega}\nabla(\zeta^q)\nabla u_i u_{i,t}dx,\
V_{u_i}=\mu_i\int_{\Omega}\nabla(\zeta^q)\nabla u_iu_i^{p}dx,\\
&W_{u_i}=\int_{\Omega}u_i^{-1}|\nabla u_i|^2\nabla u_i\nabla(\zeta^q)dx,\ R_{u_i}=\int_{\Omega}\Delta(\zeta^q)|\nabla u_i|^2dx,\ S_{u_i}=\beta\int_{\Omega}\nabla u_i\nabla(\zeta^q)u_ju_kdx\\
&\text{and}\quad Q_{u_i}=\beta\int_{\Omega}\zeta^q\nabla u_i\nabla
u_ju_kdx+\beta\int_{\Omega}\zeta^q\nabla u_i\nabla u_ku_jdx.
\end{split}
\end{equation}
By H\"older's inequality we have
\begin{equation}
\label{d-9}
\begin{split}
2\beta\int_{\Omega}\zeta^q\nabla u_i\nabla u_ju_kdx\leq\beta\int_{\Omega}\zeta^q|\nabla u_i|^2u_i^{-1}u_ju_kdx+\beta\int_{\Omega}\zeta^q|\nabla
u_j|^2u_j^{-1}u_iu_kdx.
\end{split}
\end{equation}
Based on \eqref{d-9}, we take the summation of the inequality
\eqref{d-7} from $i=1$ to $3$ and get that
\begin{equation}
\label{d-10}
\begin{split}
&A_0\sum_{i=1}^{3}\int_{\Omega}\zeta^qu_i^{-2}|\nabla u_i|^4dx+\sum_{i=1}^{3}B_i\int_{\Omega}\zeta^qu_i^{p-1}|\nabla u_i|^2dx
+\tilde{C}_1\sum_{i=1}^{3}\int_{\Omega}\zeta^q|\nabla u_i|^2u_i^{-1}u_ju_kdx\\
&\leq-\frac{N-1}{N}\sum_{i=1}^{3}\left(\frac{df_{u_i}}{dt}-P_{u_i}\right)+\frac{(N+2)d}{2N}
\sum_{i=1}^{3}Y_{u_i}+\frac{1}{N}\sum_{i=1}^{3}Z_{u_i}\\&\quad-\frac{1}{N}\sum_{i=1}^{3}V_{u_i}+\frac{d}{2}\sum_{i=1}^{3}W_{u_i}+\frac{1}{2}\sum_{i=1}^{3}R_{u_i}-\frac1N\sum_{i=1}^{3}S_{u_i},
\end{split}
\end{equation}
where
\begin{equation}
\label{d-11}
\tilde{C}_1=\frac{(N+2)d-4(N-1)}{2N}\beta.
\end{equation}
Since $2\leq N\leq5$ and $p<\frac{N(N+2)}{(N-1)^2}$, we can choose
$d>0$ such that
\begin{equation}
\label{d-12}
\max\left\{\frac{2(N-1)p}{N+2},\frac{4(N-1)}{N+2}\right\}<d<\frac{2N}{N-1}.
\end{equation}
This guarantees that $A_0, B_i, \tilde{C}_1>0~(i=1,2,3)$.

Next, by Young's inequality, we can control $Y_{u_i}, Z_{u_i}, V_{u_i}, W_{u_i}, R_{u_i}$ and $S_{u_i}$ in a similar way as \eqref{b-16}. Then we give the estimation for the terms $\int_{\Omega}\zeta^qu_{i,t}^2,~i=1,2,3$.
A direct computation shows that
\begin{equation}
\label{d-13}
\begin{split}
\int_{\Omega}\zeta^qu_{i,t}^2dx&=\int_{\Omega}\zeta^qu_{i,t}(\Delta u_i+\mu_iu_i^p+\beta u_ju_k)dx\\
&=\frac{d(g_{u_i}-f_{u_i})}{dt}+P_{u_i}-Z_{u_i}-K_{u_i}+\frac{dh}{dt}-H_{u_j}-H_{u_k}-H_{\zeta},
\end{split}
\end{equation}
where
\begin{equation}
\label{d-14}
\begin{split}
&g_{u_i}=\frac{\mu_i}{p+1}\int_{\Omega}\zeta^qu_i^{p+1}dx,\
K_{u_i}=\frac{\mu_i}{p+1}\int_{\Omega}(\zeta^q)_tu_i^{p+1}dx,\
h=\beta\int_{\Omega}\zeta^qu_iu_ju_kdx\\
&H_{\zeta}=\beta\int_{\Omega}(\zeta^q)_tu_iu_ju_kdx,\
H_{u_j}=\beta\int_{\Omega}\zeta^qu_iu_{j,t}u_kdx~\text{and}~ H_{u_k}=\beta\int_{\Omega}\zeta^qu_iu_{k,t}u_jdx.
\end{split}
\end{equation}
Using Young's estimates for the terms $P_{u_i},~i=1,2,3$, we get that for any $\varepsilon^2>0$ small, there exists $C(\varepsilon^2)>0$ such that
\begin{equation}
\label{d-15}
\begin{split}
P_{u_i}=\frac{q}{2}\int_{\Omega}\zeta^{q-1}\zeta_t|\nabla
u_i|^2dx\leq\varepsilon^2\int_{\Omega}\zeta^{q}u_i^{-2}|\nabla u_i|^4dx+C(\varepsilon^2)\int_{\Omega}\zeta^{q-2}|\zeta_t|^2u_i^{2}dx,~i=1,2,3.
\end{split}
\end{equation}
Combining \eqref{d-13}-\eqref{d-15}, we deduce that
\begin{equation}
\label{d-16}
\begin{split}
\frac{1}{2}\sum_{i=1}^3\int_{\Omega}\zeta^qu_{i,t}^2dx&\leq\sum_{i=1}^3\frac{d(g_{u_i}-f_{u_i})}{dt}+2\varepsilon^2\sum_{i=1}^3\int_{\Omega}\zeta^{q}u_i^{-2}|\nabla u_i|^4dx+\frac{dh}{dt}+C\sum_{i=1}^3\int_{\Omega}\zeta^{q-1}|\zeta_t|u_i^{p+1}dx\\
&\quad+C(\varepsilon^2)\sum_{i=1}^3\int_{\Omega}(\zeta^{q-2}\zeta_t^2+\zeta^{q-4}|\nabla\zeta|^4)u_i^2dx+\beta q\int_{\Omega}\zeta^{q-1}|\zeta_t|u_1u_2u_3dx.
\end{split}
\end{equation}
Using the assumption that $\zeta\in[0,1]$, combining the estimates
for $Y_{u_i},Z_{u_i},V_{u_i},W_{u_i},R_{u_i},S_{u_i}$ (similar as
\eqref{b-16}), \eqref{d-10} and \eqref{d-15}-\eqref{d-16}, we obtain
that
\begin{equation}\label{d-18}
\begin{split}
&(A_0-\varepsilon)\sum_{i=1}^{3}\int_{\Omega}\zeta^qu_i^{-2}|\nabla u_i|^4dx+\sum_{i=1}^{3}(B_i-\varepsilon)
\int_{\Omega}\zeta^qu_i^{p-1}|\nabla u_i|^2dx
+(\tilde{C}_1-\varepsilon)\sum_{i=1}^{3}\int_{\Omega}\zeta^q|\nabla u_i|^2u_i^{-1}u_ju_kdx\\
&\leq C\sum_{i=1}^{3}\frac{df_{u_i}}{dt}+C(\varepsilon)\sum_{i=1}^3\frac{d(g_{u_i}-f_{u_i})}{dt}+C(\varepsilon)\frac{dh}{dt}+C(\varepsilon)\sum_{i=1}^3\int_{\Omega}\zeta^{q-2}(|\zeta_t|+|\nabla\zeta|^2)u_i^{p+1}dx\\
&\quad+C(\varepsilon)\beta \int_{\Omega}\zeta^{q-2}(|\zeta_t|+|\nabla\zeta|^2)u_1u_2u_3dx
+C(\varepsilon)\sum_{i=1}^3\int_{\Omega}\zeta^{q-4}(\zeta_t^2+|\nabla\zeta|^4+|\Delta\zeta|^4)u_i^2dx.
\end{split}
\end{equation}
Since $\zeta\in\mathcal {D}(G)$, it follows that
$$f_{u_i}(t_j)=g_{u_i}(t_j)=h(t_j)=0,\quad  i=1,2,3; j=1,2.$$ Integrating the inequality \eqref{d-18} from $t_1$ to $t_2$, we obtain that
\begin{equation}
\label{d-19}
\begin{split}
&\sum_{i=1}^{3}\int_{G}\zeta^qu_i^{-2}|\nabla u_i|^4dxdt+\sum_{i=1}^{3}\int_{G}\zeta^qu_i^{p-1}|\nabla
u_i|^2dxdt+\sum_{i=1}^{3}\int_{G}\zeta^q|\nabla u_i|^2u_i^{-1}u_ju_kdxdt\\
&\leq C\sum_{i=1}^3\int_{G}\zeta^{q-2}(|\zeta_t|+|\nabla\zeta|^2)u_i^{p+1}dxdt
+C\beta\int_{G}\zeta^{q-2}(|\zeta_t|+|\nabla\zeta|^2)u_1u_2u_3dxdt\\
&\quad+C\sum_{i=1}^3\int_{G}\zeta^{q-4}(\zeta_t^2+|\nabla\zeta|^4+|\Delta\zeta|^2)u_i^2dxdt.
\end{split}
\end{equation}
Finally, the estimation for the term
$\int_{G}\zeta^q(u_{1,t}^2+u_{2,t}^2+u_{3,t}^2)dxdt$ holds from
\eqref{d-18} and \eqref{d-19}. Hence we get \eqref{d-1} and it
finishes the proof.
\end{proof}

The next lemma gives the estimation for the nonlinear terms of \eqref{auto-2}.

\begin{lemma}
\label{lem-4.2}
Assume that the assumptions of Lemma \ref{lem-2.2} hold and
$q>\frac{4p}{p-1}$. Then there exists a positive constant $C=C(N,p,q)>0$ such that
\begin{equation}
\label{d-20}
\begin{split}
&\int_{G}\zeta^q\left[\left(\mu_1u^p+\beta vw\right)^2+\left(\mu_2v^p+\beta uw\right)^2+\left(\mu_3w^p+\beta
    uv\right)^2\right]dxdt\\
&\leq C\int_{G}\zeta^{q-\frac{4p}{p-1}}
\left(|\nabla\zeta|^2+|\Delta\zeta|+|\zeta_t|\right)^{\frac{2p}{p-1}}dxdt.
\end{split}
\end{equation}
\end{lemma}

\begin{proof}
As Lemma \ref{lem-4.1} we set $(u_1, u_2, u_3)=(u,v,w)$. We deduce from the system \eqref{auto-2} that
\begin{equation}
\label{d-21}
\begin{split}
\int_{\Omega}\zeta^q\left(\mu_iu_i^p+\beta
u_ju_k\right)^2dx&=\int_{\Omega}\left(\mu_iu_i^p+\beta
u_ju_k\right)(u_{i,t}-\Delta u_i)\zeta^q\\
&=\frac{dg_{u_i}}{dt}-K_{u_i}+V_{u_i}+\mu_ip\int_{\Omega}\zeta^qu_i^{p-1}|\nabla u_i|^2dx+\frac{dh}{dt}-H_{u_j}\\
&\quad-H_{u_k}-H_{\zeta}+Q_{u_i}+S_{u_i}
\end{split}
\end{equation}
where $i=1,2,3$ and $i, j, k\in\{1, 2, 3\}$ are pairwise different. Similar to \eqref{d-19}, we infer from Lemma \ref{lem-4.1} and \eqref{d-21} that
\begin{equation}
\label{d-22}
\begin{split}
\sum_{i=1}^3\int_{G}\zeta^q\left(\mu_iu_i^p+\beta u_ju_k\right)^2dxdt\leq~&
C\sum_{i=1}^3\int_{G}\zeta^{q-2}(|\zeta_t|+|\nabla\zeta|^2)u_i^{p+1}dxdt+C\beta\int_{G}\zeta^{q-2}(|\zeta_t|+|\nabla\zeta|^2)u_1u_2u_3dxdt\\
&+C\sum_{i=1}^3\int_{G}\zeta^{q-4}(\zeta_t^2+|\nabla\zeta|^4+|\Delta\zeta|^2)u_i^2dxdt\\
\leq~&\varepsilon\left(\sum_{i=1}^3\int_{G}\zeta^qu_i^{2p}dxdt+\beta^2\int_{G}\zeta^q(u_1^{2}u_2^2+u_2^{2}u_3^2+u_1^{2}u_3^2)dxdt
\right)\\
&+C(\varepsilon)\int_{G}\zeta^{q-\frac{4p}{p-1}}(|\zeta_t|+|\Delta\zeta|+|\nabla\zeta|^2)^{\frac{2p}{p-1}}dxdt,
\end{split}
\end{equation}
where we have used the Young's inequality. Hence we get \eqref{d-20}.
\end{proof}

Now we give the proof of Theorem \ref{th1.6}.

\begin{proof}[Proof of Theorem \ref{th1.6}]
We first prove the conclusion $(i)$. For each $R>0$, we define
\begin{equation}
\label{d-23}
(u^R, v^R, w^R)=\left(R^{\frac{2}{p-1}}u(Rx, R^2t),
R^{\frac{2}{p-1}}v(Rx, R^2t), R^{\frac{2}{p-1}}w(Rx, R^2t)\right).
\end{equation}
Thus, if $(u,v,w)$ satisfies the system \eqref{auto-2}, then we get that $(u^R, v^R, w^R)$ satisfies
\begin{equation}
\label{d-24}
\begin{cases}
\ds u_t^R-\Delta u^R=\mu_1(u^R)^p+\tilde{\beta}_Rw^Rv^R,\ &(x,t)\in\mathbb{R}^{N}\times\mathbb{R},\\
\ds v_t^R-\Delta v^R=\mu_2(v^R)^p+\tilde{\beta}_Rw^Ru^R,\
&(x,t)\in\mathbb{R}^{N}\times\mathbb{R},\\
\ds w_t^R-\Delta w^R=\mu_3(w^R)^p+\tilde{\beta}_Ru^Rv^R,\
&(x,t)\in\mathbb{R}^{N}\times\mathbb{R},
\end{cases}
\end{equation}
where $\tilde{\beta}_R=\beta R^{\frac{2(p-2)}{p-1}}$. Repeating the arguments in Lemma \ref{lem-4.1}-\ref{lem-4.2} we derive from H\"older's and Young's inequalities that for any $\varepsilon>0$ small, there exists $C(\varepsilon)>0$ such that
\begin{equation}
\label{d-25}
\begin{split}
&\int_{G}\left[\left(\mu_1(u^R)^p+\tilde{\beta}_Rw^Rv^R\right)^2+\left(\mu_2(v^R)^p+\tilde{\beta}_Rw^Ru^R\right)^2+\left(\mu_3(w^R)^p+\tilde{\beta}_Ru^Rv^R\right)^2\right]dydt\\&\leq C\tilde{\beta}_R\int_{G}\zeta^{q-2}(|\zeta_t|+|\nabla\zeta|^2)
u^Rv^Rw^Rdydt+C\int_{G}\left(|\Delta\zeta|^2+|\nabla\zeta|^4+|\zeta_t|^2\right)\zeta^{q-4}((u^R)^2+(v^R)^2+(w^R)^2)dydt\\
&\quad+C\int_{G}\zeta^{q-2}(|\zeta_t|+|\nabla\zeta|^2)((u^R)^{p+1}+(v^R)^{p+1}+(u^R)^{p+1})dydt\\
&\leq\varepsilon\tilde{\beta}_R^2\int_{G}\zeta^q((u^R)^2(v^R)^2+(u^R)^2(w^R)^2+(w^R)^2(v^R)^2)dydt+\varepsilon\int_{G}\zeta^q((u^R)^{2p}+(u^R)^{2p}+(w^R)^{2p})dydt\\
&\quad+C(\varepsilon)\int_{G}\zeta^{q-\frac{4p}{p-1}}\left(|\Delta\zeta|+|\nabla\zeta|^2+|\zeta_t|\right)^{\frac{2p}{p-1}}dydt,
\end{split}
\end{equation}
where $q>\frac{4p}{p-1}$, and $C, C(\varepsilon)$ are independent of $R$. We can see the first two terms can be controlled by the left-hand side, then
\begin{equation}
\label{d-26}
\begin{split}
&\int_{G}\left[\left(\mu_1(u^R)^p+\tilde{\beta}_Rw^Rv^R\right)^2+\left(\mu_2(v^R)^p+\tilde{\beta}_Rw^Ru^R\right)^2+\left(\mu_3(w^R)^p+\tilde{\beta}_Ru^Rv^R\right)^2\right]dydt\\
&\leq~C(\varepsilon)\int_{G}\zeta^{q-\frac{4p}{p-1}}\left(|\Delta\zeta|+|\nabla\zeta|^2+|\zeta_t|\right)^{\frac{2p}{p-1}}dydt\leq C.
\end{split}
\end{equation}
Hence we deduce from \eqref{d-26} that
\begin{equation}
\label{d-30}
\begin{split}
&\int_{-R^2}^{R^2}\int_{|y|<R}\left[\left(\mu_1u^p+\beta uv\right)^2+\left(\mu_2v^p+\beta wu\right)^2+\left(\mu_3w^p+\beta uv\right)^2\right]dydt\\
&=R^{N+2-\frac{4p}{p-1}}\int_{-1}^{1}\int_{|y|<1}\left[\left(\mu_1(u^R)^p+\tilde{\beta}_Ru^Rv^R\right)^2+\left(\mu_2(v^R)^p+\tilde{\beta}_Ru^Rw^R\right)^2+\left(\mu_3(w^R)^p+\tilde{\beta}_Ru^Rv^R\right)^2\right]dydt\\
&\leq CR^{N+2-\frac{4p}{p-1}}.
\end{split}
\end{equation}
Since $p<\frac{N(N+2)}{(N-1)^2}<p_S(N)$, by letting $R\rightarrow\infty$, we conclude that the right-hand side of \eqref{d-30} converges to $0$. As a consequence, we deduce that $u=v=w=0$.

The second conclusion of Theorem \ref{th1.6} can be proved following the same arguments as in Theorem \ref{th1.1}-$(ii)$, so we omit the details here.
\end{proof}

\subsection{Periodic solutions}
In this subsection we focus on the existence of periodic solutions to the three coupled system \eqref{auto-13}. Let $X=BUC(\Omega\times(0,T))$ denote the space of the bounded uniformly continuous functions equipped with the $L^\infty$-norm $\|\cdot\|_\infty$. For each $z$, we denote $z^+(x,t)=\max\{z(x,t),0\}$. By slightly abuse of notation, we use $\|\cdot\|_\infty$ to denote the norm for both
$L^\infty(\Omega\times(0,T))$ and $L^\infty(\Omega)$.

As Theorem \ref{th1.3} we introduce the following homotopy problem and show the universal bounds for the corresponding periodic solutions
\begin{equation}
\label{d-31}
\begin{cases}
\ds u_t-\Delta u=\theta(\mu_1u^p+\beta uv)+(1-\theta)(\lambda u+u^2),\ &x\in\Omega,\ t\in(0, \infty),\\
\ds v_t-\Delta v=\theta(\mu_2v^p+\beta uw)+(1-\theta)(\lambda
v+v^2),\ &x\in\Omega,\
t\in(0, \infty),\\
\ds w_t-\Delta w=\theta(\mu_3w^p+\beta uv)+(1-\theta)(\lambda
w+w^2),\ &x\in\Omega,\ t\in(0, \infty),\\
u(x,t)=v(x,t)=w(x,t)=0,\quad &\text{on}\ \partial\Omega\times(0,
\infty),
\end{cases}
\end{equation}
where $\theta\in[0,1]$ and $\lambda>0$ is a constant. To accomplish
this we prove Liouville type results for the following
system
\begin{equation}
\label{d-32}
\begin{cases}
\ds u_t-\Delta u=\theta\left(\mu_1u^2+\beta vw\right)+(1-\theta)u^2,\ &x\in\Omega,\ t\in\mathbb{R},\\
\ds v_t-\Delta v=\theta\left(\mu_2v^2+\beta
uw\right)+(1-\theta)v^2,\ &x\in\Omega,\ t\in\mathbb{R},\\
\ds w_t-\Delta w=\theta\left(\mu_3w^2+\beta
uv\right)+(1-\theta)w^2,\ &x\in\Omega,\ t\in\mathbb{R},
\end{cases}
\end{equation}
where $\mu_1, \mu_2, \mu_3, \beta>0$ and $\theta\in[0,1]$. By using the same proof of Theorem \ref{th1.6}-$(i)$, we can
obtain the following result.

\begin{lemma}
\label{lem4.3}
Suppose that $N\leq5,$ $\mu_1, \mu_2, \mu_3, \beta>0,$ $\theta\in[0,1]$ and $\Omega=\mathbb{R}^{N}$. Then the system \eqref{d-32} has no nontrivial nonnegative classical solution.
\end{lemma}

Based on the above lemma, we are able to prove the universal bounds for the periodic solutions of \eqref{d-31}.

\begin{lemma}
\label{lem4.4}
Suppose that $N\leq 5$, $\mu_1, \mu_2, \mu_3, \beta>0, \theta\in[0,1]$ and $2<p<p_B(N)$. Then there exists a constant $C>0$ such that any positive $T$-periodic solution $(u,v)$ of \eqref{d-31} satisfies
\begin{equation}
\label{d-33}
u(x,t)+v(x,t)+w(x,t)\leq C\quad\text{for\ all}\ (x,t)\in\Omega\times(0,\infty).
\end{equation}
\end{lemma}

\begin{proof}
This follows the same proof of Lemma \ref{lem3.3} and we omit the details.
\end{proof}

Now we give the proof of Theorem \ref{th1.7}.

\begin{proof}[Proof of Theorem \ref{th1.7}.]
Let $\mathcal {T}$ be the compact map by Lemma \ref{lem3.1}. We define the compact operator $\mathcal {L}: X\times X\times X\rightarrow X\times X\times X$ by
\begin{equation}
\label{d-41}
\mathcal {L}(u,v)=\left(\mathcal {T}(\mu_1u^p+\beta vw), \mathcal {T}(\mu_2v^p+\beta uw), \mathcal {T}(\mu_3w^p+\beta uv)\right).
\end{equation}
It is clear that the nontrivial fixed point of \eqref{d-41} corresponds to a nonnegative periodic solutions of \eqref{auto-13}. Indeed, if $(u,v,w)\neq(0,0)$ is a fixed point of $\mathcal {T}$, then we see $(u,v,w)=(z^+, h^+, m^+)$ and $(z, h,m)$ is a $T$-periodic solution of
\begin{equation}
\label{d-42}
\begin{cases}
\ds z_t-\Delta z=\mu_1(z^+)^p+\beta m^+h^+,\ &x\in\Omega,\ t\in(0, \infty),\\
\ds h_t-\Delta h=\mu_2(h^+)^p+\beta m^+z^+,\ &x\in\Omega,\
t\in(0, \infty),\\
\ds m_t-\Delta m=\mu_3(m^+)^p+\beta z^+h^+,\ &x\in\Omega,\
t\in(0, \infty),\\
z(x,t)=h(x,t)=m(x,t)=0,\quad &(x,t)\in \partial\Omega\times(0,
\infty).
\end{cases}
\end{equation}
Using the maximum principle we know that $z, h, m\geq0$. Furthermore, we claim that $z, h, m>0$. In fact, if $z\equiv0$, then $h$ and $m$ satisfies
\begin{equation}
\label{d-421}
h_t-\Delta h=\mu_2h^p,\quad
m_t-\Delta m=\mu_3m^p
\end{equation}
respectively. From the proof of Lemma \ref{lem3.3} we get that  any periodic solution $h$, $m$ of \eqref{d-421} are bounded. That is, $h(x,t),m(x,t)\leq C$, where $C>0$. We shall use $m$ as an example to show that $m\equiv0$. Multiplying the equation of $m$ in \eqref{d-421} by $\psi$ (see Lemma \ref{lem3.1} for the definition of $\psi$) and integrating over $\Omega$, we get
\begin{equation*}
\frac{d}{dt}\int_{\Omega}m(\cdot,t)\psi dx=\int_{\Omega}\left(-\lambda_1^T+\mu_3m^{p-1}\right)\psi mdx<0
\end{equation*}
for $\mu_3\geq0$ sufficiently small. It means that the function
$t\mapsto\int_{\Omega}m(\cdot,t)\psi dx$ is strictly decreasing in time. Together with the periodicity of $m$ we derive that $m\equiv0$. Similarly $h\equiv0$, and it contradicts $(z,m,h)$ is nontrivial. Hence $z$ can not be $0$ identically. By strong maximum principle we get $z>0$. By the same argument we get both $h$ and $m$ are strictly positive.

In the end, following the proof of Theorem \ref{th1.3}, we know that the Leray-Schauder degree of $F(r)=deg(I-\mathcal {T}, B_r, 0)=1$ for $r>0$ small and $F(r)=0$ for large $r>0$, and it implies the existence of the periodic solutions to \eqref{auto-13}. Hence, we finish the proof.
\end{proof}

\vspace{2cm}
\begin{center}
{\bf Acknowledgement}
\end{center}

The second author is partially supported by NSFC No. 11971202, the
Outstanding Young foundation of Jiangsu Province No. BK2020010118
and the Six big talent peaks project in Jiangsu Province(XYDXX-015);
The third author is partially supported by NSFC No.11801550 and NSFC
No.11871470.

\end{document}